\documentclass[11pt]{amsart}

\usepackage[OT1,T1]{fontenc}
\usepackage{times}
\usepackage{calligra}
\usepackage{aurical}
\usepackage{mathpazo}
\usepackage{ulem}
\usepackage{amsthm}
\usepackage[mathscr]{euscript}
\usepackage{ytableau}
\usepackage{young}
\usepackage{youngtab}

\usepackage{stmaryrd}
\usepackage{amssymb}
\usepackage{amsmath}
\usepackage{titletoc}
\usepackage{extarrows}
\usepackage{varioref}
\usepackage{bm}
\usepackage{amsthm}
\usepackage[all,cmtip]{xy}
\usepackage{tikz-cd}
\usepackage{color} 
\usepackage{soul}
\setulcolor{red}
\sethlcolor{yellow}
\setstcolor{blue}
\usepackage{amscd}
\usepackage{xfrac}    

\usepackage{amsfonts}
\usepackage[colorlinks,linkcolor=blue,citecolor=blue, pdfstartview=FitH]{hyperref}
\usepackage{graphicx}
\usepackage{geometry}
\usepackage{cite}

\numberwithin{equation}{section}

\theoremstyle{definition}
\newtheorem{thm}{Theorem}[section]
\newtheorem{theorem}[thm]{Theorem}

\newtheorem{lemma}[thm]{Lemma}
\newtheorem{corollary}[thm]{Corollary}
\newtheorem{proposition}[thm]{Proposition}

\newtheorem{remark}[thm]{Remark}

\newtheorem{definition}[thm]{Definition}
\newtheorem{claim}[thm]{Claim}
\newtheorem{assumption}[thm]{Assumption}

\newtheorem{example}[thm]{Example}

\newtheorem{defn-thm}[thm]{Definition-Theorem}
\newtheorem{propdef}{Proposition-Definition}

\newenvironment{observe}{\noindent\textcolor{blue}{\textit{Observation}}.}{\hfill \textcolor{blue}{$\blacktriangleleft$}\par}

\newcommand{\sA}{{\mathcal A}}

\newcommand{\sE}{{\mathcal E}}
\newcommand{\sF}{{\mathcal F}}
\newcommand{\sG}{{\mathcal G}}
\newcommand{\sH}{{\mathcal H}}

\newcommand{\sK}{{\mathcal K}}
\newcommand{\sL}{{\mathcal L}}

\newcommand{\sN}{{\mathcal N}}
\newcommand{\sO}{{\mathcal O}}
\newcommand{\sP}{{\mathcal P}}
\newcommand{\sQ}{{\mathcal Q}}

\newcommand{\sS}{{\mathcal S}}

\newcommand{\sU}{{\mathcal U}}

\newcommand{\g}{{\mathfrak g}}

\newcommand{\gT}{{\mathfrak T}}

\newcommand{\ssA}{{\mathscr A}}

\newcommand{\ssE}{{\mathscr E}}

\newcommand{\ssP}{{\mathscr P}}

\newcommand{\A}{\mathbb{A}}

\newcommand{\N}{{\mathbb N}}

\newcommand{\R}{{\mathbb R}}
\newcommand{\Z}{{\mathbb Z}}

\newcommand{\GL}{\operatorname{GL}}
\newcommand{\SL}{\operatorname{SL}}

\newcommand{\Ext}{\operatorname{Ext}}

\newcommand{\Hom}{\operatorname{Hom}}

\newcommand{\Ker}{\operatorname{Ker}}

\newcommand{\Pic}{\operatorname{Pic}}

\newcommand{\Spec}{\operatorname{Spec}}

\newcommand{\Sch}{\operatorname{\mathbf{Sch}}}

\newcommand{\tr}{{\operatorname{tr}}}

\newcommand{\End}{{\operatorname{End}}}

\newcommand{\Higgs}{\operatorname{Higgs}}

\newcommand{\Res}{\operatorname{Res}}

\newcommand{\id}{{\operatorname{id}}}

\newcommand{\btheorem}{\begin{theorem}}
	\newcommand{\etheorem}{\end{theorem}}
\newcommand{\bproposition}{\begin{proposition}}
	\newcommand{\eproposition}{\end{proposition}}
\newcommand{\bdefinition}{\begin{definition}}
	\newcommand{\edefinition}{\end{definition}}
\newcommand{\bcorollary}{\begin{corollary}}
	\newcommand{\ecorollary}{\end{corollary}}
\newcommand{\bproof}{\begin{proof}}
	\newcommand{\eproof}{\end{proof}}
\newcommand{\bremark}{\begin{remark}}
	\newcommand{\eremark}{\end{remark}}
\newcommand{\eexample}{\end{example}}
\newcommand{\bexample}{\begin{example}}
\newcommand{\elemma}{\end{lemma}}
\newcommand{\blemma}{\begin{lemma}}
\newcommand{\bobserve}{\begin{observe}}
	\newcommand{\eobserve}{\end{observe}}

\renewcommand{\bar}{\overline}

\renewcommand{\phi}{\varphi}

\newcommand{\ee}{\end{eqnarray*}}
\newcommand{\be}{\begin{eqnarray*}}

\newcommand{\beq}{\begin{equation}}
	\newcommand{\eeq}{\end{equation}}

\newcommand{\bd}{\begin{enumerate}}
	\newcommand{\ed}{\end{enumerate}}
\newcommand{\bti}{\begin{tikzcd}}
	\newcommand{\eti}{\end{tikzcd}}

\renewcommand{\hat}{\widehat}
\renewcommand{\tilde}{\widetilde}
\newcommand{\pdeg}{\text{par-}deg}
\newcommand{\pmu}{\text{par-}\mu}
\def\pt{{\scriptscriptstyle\bullet}}

\newcommand{\PGL}{\operatorname{PGL}}

\newcommand{\M}{{\textbf{M}}}
\begin{document}

\title{Parabolic Hitchin Maps and Their Generic Fibers
}




	\author{Xiaoyu Su$^1$} 
\address{$^1$ Yau Mathematical Sciences Center, Beijing, 100084, China.}
\email{suxiaoyu@mail.tsinghua.edu.cn}

\author{Xueqing Wen$^2$ }
\address{$^2$ Yau Mathematical Sciences Center, Beijing, 100084, China.}
\email{xueqingwen@mail.tsinghua.edu.cn}

\author{Bin Wang$^3$ }
\address{$^3$ Steklov Mathematical Institute of Russian Academy of Science, Moscow, 119991, Russia.
}
\email{binwang@mi-ras.ru}



\maketitle

\begin{abstract}
	
	We set up a BNR correspondence for  moduli spaces of Higgs bundles over a curve with a parabolic structure over any algebraically closed field. This leads to a concrete description of generic fibers of the associated strongly parabolic Hitchin map. We also show that the global nilpotent cone is equi-dimensional with half dimension of the total space. As a result, we prove the flatness and surjectivity of this map and the existence of very stable parabolic vector bundles.
	\vspace{10pt}
	
\end{abstract}

%

\section{Introduction}

Hitchin  \cite{Hit87S} introduced the map now named after him,  and  showed that it defines  a completely integrable system in the complex-algebraic  sense.  Subsequently in  \cite{BNR}, Beauville, Narasimhan and Ramanan constructed a correspondence --- indeed nowadays referred to as the BNR correspondence --- which among other things characterizes  a generic fiber of the Hitchin map as a compactified Jacobian. 

Our paper is concerned with a parabolic version of the BNR correspondence. Roughly speaking, we show that a generic fiber of the strongly parabolic Hitchin map can be identified with the Picard variety (of certain degree) of the normalization of the corresponding spectral curve. Since we consider the parabolic structures which need not be full flags, generic spectral curves are singular in general. We analyze singularities of generic spectral curves via the toric resolution, and thus, among other things, the genericity condition is closely related with Newton polygons of local equations determined by points in the Hitchin base space which forms an open subvariety. See Assumption \ref{generic on boud} and Definition \ref{def:genericity condition} for more details. Though Baraglia, Kamgarpour and Varma \cite{BKV18} proved that generic fibers are Abelian varieties for parahoric Hitchin systems with the structure group simply-connected and simple over $\mathbb{C}$, our results yield a modular interpretation of generic fibers in the case of $\GL_{r}$ which can be applied to study the Langlands duality between parabolic Hitchin systems.

Besides, we show that the parabolic global nilpotent cone is equi-dimensional and has half dimension of the total space. The strategy is to study deformations of nilpotent Higgs bundles within the global nilpotent cone,  which among other things, can be used to prove the existence of very stable parabolic vector bundles, i.e., parabolic bundles that do not admit nonzero nilpotent Higgs fields. All these results are in the setting of algebraic geometry. By this we mean that we work over an arbitrary algebraically closed field $k$. 

To be concrete,  let us fix a smooth projective curve $X$ over $k$ of genus $g(X)\geq 2$ and  a finite subset $D\subset X$, which we shall also regard as a reduced effective divisor  on $X$. We also fix a positive integer $r$ which will be the rank of  vector bundles on $X$ that we shall consider (but if  $k$ has  characteristic $2$, we shall assume $r\geq 3$ in order to avoid issues involving ampleness, please see the Appendix) and we specify for each
$x\in D$ a finite sequence $m^\pt (x)=(m^1(x), m^2(x), \dots, m^{\sigma_x}(x))$ of positive integers summing up to $r$. We refer to these data as a \emph{quasi-parabolic structure}; let us denote this simply by $P$ and we call $m^\pt (x)$ the \emph{Levi type} of $P$.

A \emph{quasi-parabolic  vector bundle of type $P$} is then a rank $r$ vector bundle $E$ on $X$ which for every  $x\in D$ is endowed with a filtration $E|_{x}=F^0(x)\supset F^1(x)\supset \cdots \supset F^{\sigma_x}(x)=0$ such that $\dim F^{j-1}(x)/ F^{j}(x)=m^j(x)$. 
A  \emph{strongly parabolic Higgs field} on such a bundle is an $\sO_X$-homomorphism  $\theta : E\to E\otimes_{\sO_X} \omega_X(D)$ with the property that it takes each $F^j(x)$ to $F^{j+1}(x)\otimes_{\sO_X}\omega_{X}(D)|_{x}$. We call it a \emph{weakly parabolic Higgs field}, if it only takes 
$F^j(x)$ to $F^{j}(x)\otimes_{\sO_X}\omega_{X}(D)|_{x}.$ A weakly parabolic Higgs field $\theta$ has a characteristic polynomial 
with coefficients as an element of $\mathcal{H}:=\prod_{j=1}^rH^0 (X, (\omega(D))^{\otimes j})$, and in the total space of $\omega_{X}(D)$, the characteristic polynomial  itself defines
the \emph{spectral curve} that is finite over $X$.

With the help of the Geometric Invariant Theory, one can construct coarse moduli spaces of such objects, but this requires the
"polarization data", which in the present context take the form of a \emph{weight function $\alpha$} which assigns to every $x\in D$ a set of real numbers 
$0\leq \alpha_1(x)<\cdots <\alpha_{\sigma_x}(x)<1$. As we will recall later, this gives rise to notions of parabolic structures and corresponding stability conditions. 
And these lead to the existence of coarse moduli spaces as quasi-projective varieties parametrizing the isomorphism classes of $\alpha$-stable objects of type $P$: for parabolic vector bundles we get $\M_{P,\alpha}$, for  weakly  parabolic Higgs bundles we get $\Higgs^W_{P,\alpha}$ and for strongly parabolic Higgs bundles  we get $\Higgs_{P,\alpha}$, the latter being contained in $\Higgs^W_{P,\alpha}$ as a  closed subvariety.
If we choose generic $\alpha$, the notions of semistability and stability coincide, and the corresponding coarse moduli spaces are nonsingular. 

By assigning to a Higgs field the  coefficients of its  characteristic polynomial, we obtain the weak parabolic Hitchin map $h^W_{P,\alpha}: \Higgs^W_{P,\alpha}\to  \mathcal{H}$. 
We prove that $h^W_{P,\alpha}$ is flat by showing that each connected component of the generic fiber of $h^{W}_{P,\alpha}$ is a torsor  of the  Jacobian variety of the corresponding spectral curve and compute the number of connected components. The number of connected components of fibers is obtained by Logares and Martens \cite{LM10} in a more general setting over $\mathbb{C}$.

Let  $\mathcal{H}_{P}$ be the  image of $\Higgs_{P,\alpha}$. We denote the resulting morphism as
$$h_{P,\alpha}: \Higgs_{P,\alpha}\to \mathcal{H}_{P},$$ and refer it as the strongly parabolic Hitchin map. We characterize $\mathcal{H}_{P}$ as an affine subspace of $\sH$ (this was obtained earlier by Baraglia and Kamgarpour \cite{BK18}) and prove essentially  that
$h_{P,\alpha}$ has all the properties that one would hope for, i.e., $h_{P,\alpha}$ is flat projective and the generic fibers of $h_{P,\alpha}$ are identified as Picard varieties of normalized spectral curves.

We have a commutative diagram 
\[\begin{tikzcd}
	\Higgs_{P,\alpha}\arrow[d,hook]\arrow[r,"h_{P,\alpha}"]& \mathcal{H}_{P}\arrow[d,hook]\\
	\Higgs^W_{P,\alpha}\arrow[r,"h^W_{P,\alpha}"]& \mathcal{H}
\end{tikzcd},
\]
but beware that this is not Cartesian unless all the $m^j(x)$ are equal to $1$, i.e., a full flag.
We give a concrete description of generic fibers of $h_{P,\alpha}$ by obtaining the parabolic BNR correspondence in this setting,
which, roughly speaking, amounts to (see Theorem \ref{parabolic BNR}):

\begin{theorem}[Parabolic BNR Correspondence]\label{main02}
	
	There is a one to one correspondence between: 
	$$\left\{ \begin{array}{c}\text{isomorphism classes of strongly parabolic Higgs bundles of a fixed degree}\\ \text{with a prescribed generic characteristic polynomial}    	
	\end{array} \right\} $$   	
	and      	 
	\[ \left\{\begin{array}{c}\text{line bundles over the corresponding normalized spectral curve}\\ \text{with a fixed degree determined by the parabolic data}\end{array} \right\}.\]
	For the generic condition, see Definition \ref{def:genericity condition}. In particular, generic fibers of $h_{P,\alpha}$ are connected.   	
\end{theorem}

As we mentioned at the beginning of the introduction, generic spectral curves are singular in general unless all quasi-parabolic structures come from full flags. To build up such a parabolic BNR correspondence, we first study the local structure of singularities of generic spectral curves, which has an interesting relation with parabolic structures. The ramification of normalized spectral curves in some sense corresponds to conjugate partitions, see Corollary \ref{ramification}. Then we show that parabolic structures make it possible to lift the coherent module structure of a Higgs bundle over the spectral curve to a coherent module structure over the normalization. The main technical input is concluded in Subsection \ref{local analysis}.   

Furthermore, we construct the complex that bounds the deformation inside an open dense subset of the global nilpotent cone and prove that (see also Theorem \ref{main 1} and Proposition \ref{dimestimate1}):
\begin{theorem}\label{dimestimate01}
	The parabolic global nilpotent cone $\sN il_{P}:=h_{P,\alpha}^{-1}(0)$ is equi-dimensional and $\dim\sN il_{P}=1/2\dim\Higgs_{P,\alpha}$.
\end{theorem}

Then by the miraculous flatness criterion for maps between regular varieties,  for generic parabolic weight $\alpha$, the strongly parabolic Hitchin map $h_{P,\alpha}$ is flat (see Theorem \ref{flat}). 

Besides, by Theorem \ref{dimestimate01} and using systems of Hodge bundles, we prove the existence of very stable parabolic bundles (see Theorem \ref{very stable}).
\begin{theorem}
	The set of very stable parabolic bundles contains a non-empty Zariski open set in the moduli of stable parabolic bundles $\M_P$.
\end{theorem}
Let us now indicate how this relates to previous works.  
After the fundamental work of Hitchin and  Beauville-Narasimhan-Ramanan mentioned above,  several papers investigated various properties of the  Hitchin  map over the complex field, for example in \cite{Lau88}, \cite{Fal93}, \cite{Gin01}, \cite{GO13}. Nitsure \cite{N91} constructed the moduli space of (semi-)stable $L$-twisted Higgs bundles over an algebraically closed field and showed the properness of Hitchin maps. 
In the parabolic setting, Yokogawa \cite{Yo93C,Yo95} constructed the moduli space of (semi-)stable strongly/weakly parabolic Higgs bundles and then proved that the weakly parabolic Hitchin map is proper. His construction works over any algebraically closed field. We should also point out that the classical Beauville-Narasimhan-Ramanan correspondence also holds in positive characteristics, see \cite[Corollary 4.5]{BB07}.

Bottacin and Markman in \cite{Botta95,Markman94} studied the symplectic geometry of the moduli space of $L$-twisted Higgs pairs. They defined a Poisson structure on the moduli space and showed that it forms an algebraically completely integrable system (in a generalized sense) under the $L$-twisted Hitchin map and fibered by Jacobians of spectral curves. Bottacin \cite{Botta95} also considered the symplectic geometry of the cotangent bundle of the moduli space of parabolic vector bundles and gave an explicit description of the canonical symplectic structure via deformation theory which is an analogous result of Biswas and Ramanan's in \cite{BR94}. We note here that an $\omega_{X}(D)$-twisted Higgs bundle can be viewed as a weakly parabolic Higgs bundle with trivial parabolic structure. On the other hand, the ordinary Higgs bundles can be viewed as the strongly parabolic Higgs bundles with trivial parabolic structure. Notice that the convention of ``generic weights" cannot be made for the trivial parabolic situation\footnote{We thank the referee for pointing out this.} as long as one can restrict to the smooth part to study the symplectic geometry.

We should also mention the deep results of Ngo \cite[Section 4]{Ngo10} on generic fibers of $L$-twisted $G$-Hitchin systems over a finite field. To be more precise, in \cite[Section 4.15]{Ngo10}, over an open dense subvariety (denoted as $\ssA^{\text{ani}}$ there) of the Hitchin base, Ngo constructed a homomorphism between a product of affine Springer fibers and Picard stacks of torsors of a smooth commutative group scheme on the base curve with the corresponding Hitchin fiber. As a result, Ngo could count points of the Hitchin fiber over finite fields for $L$-twisted $G$-Higgs bundles. In the $G=\GL_{n}$ case, the Hitchin system there can be treated as $L$-twisted Hitchin system with trivial weak parabolic structure as we mentioned above. Moreover, over a particular open subvariety (denoted as $\ssA^{\diamond}$ there), the Higgs bundles in each fiber are regular Higgs bundles, and the Hitchin fibers are isomorphic to Picard stacks of torsors of a smooth commutative group scheme. In the strongly parabolic case, we prove a similar geometric result. That is to say, over an open subvariety of the Hitchin base, each fiber is a Picard variety of normalized spectral curves. In fact, due to the existence of strongly parabolic structures at marked points, the Higgs field can never be regular at the marked points (unless the strongly parabolic structures come from Borel subalgebras). Our results show that at marked points, for generic strongly parabolic Higgs bundles, their Higgs fields are Richardson, which are as "regular" as possible in parabolic subalgebras, please see Subsection \ref{subsec:regularity} for more details.

Logares and Martens \cite{LM10}  studied the generic fibers of weakly parabolic Hitchin maps over the complex field (which, however, they referred to as parabolic Hitchin maps in \cite{LM10}). They constructed a Poisson structure on $\Higgs^W_{P,\alpha}$ and proved that $h^W_{P,\alpha}$ is an integrable system in the Poisson sense. In Section \ref{section5}, we determine the component group of the generic fibers of the weakly parabolic Hitchin map $h^W_{P,\alpha}$ in the spirit of BNR correspondence instead of using the Poisson structure. 

Scheinost and Schottenloher \cite{SS95}, also working over $\mathbb{C}$, defined the strongly parabolic Hitchin map $h_{P,\alpha}$ (which they referred to as the Hitchin map for parabolic Higgs bundles see \cite[Definition 5.5]{SS95} and their parabolic Higgs bundles are strongly parabolic Higgs bundles in this paper) and proved that $h_{P,\alpha}$ is an algebraically completely integrable system. Baraglia, Kamgarpour and Varma \cite{V16,BK18, BKV18} generalized this to a $G$-parahoric Hitchin system, where $G$ can be a simple, simply connected algebraic group over $\mathbb{C}$. Our method here works over an algebrically closed field, and yields a more precise description of generic fibers in the case $G=\GL_{r}$ which is a modular interpretation, i.e., a generic fiber is a Picard variety. 
\\

We close this section by describing  how this paper is organized. In Section \ref{sec 2}, we recall the parabolic setting and  review the properties of the moduli spaces $\M_{P,\alpha}$, $\Higgs_{P,\alpha}$ and $\Higgs^W_{P,\alpha}$.  In Section \ref{sec 3}, we recall the construction of  the (weakly/strongly) parabolic Hitchin maps $h^W_{P,\alpha}$ and $h_{P,\alpha}$ and determine the corresponding parabolic Hitchin base space $\sH_P$ as in \cite{BK18}.  In Section \ref{sec bnr}, we set up the parabolic BNR correspondence (Theorem \ref{main02}) and determine the generic fibers of a strongly parabolic Hitchin map. In Section \ref{section5}, we do the same for a weakly parabolic Hitchin map. And finally, in Section \ref{sec 6}, we compute the dimension of parabolic nilpotent cones and prove Theorem \ref{flat}. We also prove the existence of very stable parabolic vector bundles. As an application, we use a co-dimension estimate to give an embedding of conformal blocks into theta functions.

\begin{remark}
	Unfortunately conventions regarding parabolic Higgs bundles vary in the literature.  For example, a weakly parabolic Higgs bundle here is referred to as a ``parabolic Higgs bundle" in \cite{LM10}, a strongly parabolic Higgs bundle is referred to as a ``parabolic Higgs bundle" in \cite{BR94} and  \cite{SS95}. In order not to cause confusions,  we use ``strongly/weakly parabolic" throughout the paper.
\end{remark}

\section{Strongly and weakly parabolic Higgs bundles and parabolic Hitchin maps}\label{sec 2}

\subsection{Parabolic vector bundles}   
We use the notions introduced above. In particular, we fix $X$ and a set of quasi-parabolic data $P=(D, \{m^\pt(x)\}_{x\in D})$.
We denote by $P_x\subseteq \GL_r=G$ the standard parabolic subgroup corresponding to $\{m^j(x)\}$. We also fix a weight function $\alpha=\{\alpha_{\pt}(x)\}_{x\in D}$ and call $(P, \alpha)$ a parabolic structure.
We fix a positive integer $r$ and let $E$ be a rank $r$ vector bundle over $X$ endowed with a quasi-parabolic structure of type $P$.

\begin{remark}\label{rem:convention}
	From now on, we will use calligraphic letters $\sE, \sF, \ldots$ to denote parabolic bundles of a given type (with certain quasi-parabolic structure), and use the normal upright Roman letters $E, F,\ldots$ to denote the underlying vector bundles. We will also consider a local version (where $X$ is replaced by the spectrum of a complete DVR), then $D$ will be the closed point, and we will write $\sigma$, $\{m^{j}\}_{j=1}^{\sigma}$ and  $\{\alpha_{j}\}_{j=1}^{\sigma}$ instead.
\end{remark}

Let a parabolic vector bundle $\sE$  on $X$ be given. Then every $\sO_X$-coherent submodule $F$ of $E$ inherits from $E$ a quasi-parabolic structure  so that it may be regarded as  a parabolic sheaf $\sF$. Note that the weight function $\alpha$ for $\sE$ determines one for $\sF$. Similarly, for any line bundle $L$ on $X$, we have a natural parabolic structure on $E\otimes_{\sO_X}L$, which we then denote by $\sE\otimes_{\sO_X}L$. For more details, please refer to \cite{Yo93C}.

An endomorphism of the parabolic bundle $\sE$ is a vector bundle endomorphism of $E$ which preserves the filtrations $F^\pt(x)$. We call this a \emph{strongly parabolic endomorphism} if it takes $F^i(x)$ to $F^{i+1}(x)$ for all $x\in D$ and $i$. We denote the subspaces of $\End_{\sO_X}(E)$ defined by these properties as
$$
ParEnd(\sE) \text{  resp.\  } SParEnd(\sE).
$$ 
Similarly we can define the sheaf of parabolic endomorphisms and sheaf of strongly parabolic endomorphisms, denoted by $\sP ar\sE nd(\sE)$ and  $\sS\sP ar\sE nd(\sE)$ respectively.

\begin{remark}
	Following \cite{Yo95}, we have 
	\beq\label{(3.4)}  \sP ar \sE nd(\sE)^{\vee}\cong \sS \sP ar\sE nd(\sE)\otimes_{\sO_X}\sO_X(D).
	\eeq
\end{remark}

We now define the \emph{parabolic degree (or $\alpha$-degree)} of $\sE$ to be 
\[\pdeg(\sE):=\deg(E)+\sum_{x\in D}\sum_{j=1}^{\sigma_x}\alpha_{j}(x)m^j(x).\]
And the \emph{parabolic slope or $\alpha$-slope} of $\sE$ is given by \[\pmu(\sE)=\frac{\pdeg(\sE)}{r}.\]

\begin{definition}
	A parabolic vector bundle $\sE$ is said to be stable(resp. semistable),  if for every proper coherent $\sO_X$-submodule $F\subsetneq E$ , we have 
	\[
	\pmu(\sF)<\pmu(\sE)\ (\text{resp.}\leq),
	\] 
	where the parabolic structure on $\sF$ is  inherited from $\sE$. 
\end{definition}

There exists a coarse moduli space for semistable parabolic vector bundles of rank $r$ with fixed quasi-parabolic type $P$ and weights $\alpha$. For the constructions and properties, we refer the readers to \cite{MS80,Yo93C,Yo95}. We denote the moduli space by $\M_{P,\alpha}$ (the stable locus is denoted by $\M_{P,\alpha}^s$). $\M_{P,\alpha}$ is a normal projective variety of dimension (see \cite[Theorem 4.1]{MS80})
\begin{align*}
	\dim(\M_{P,\alpha})&=(g-1)r^2+1+\sum\limits_{x\in D}\frac{1}{2}(r^2-\sum_{j=1}^{\sigma_x}(m^{j}(x))^2)\\
	&=(g-1)r^2+1+\sum\limits_{x\in D}\dim(G/P_{x}).
\end{align*}

\subsection{Strongly and weakly parabolic Higgs bundles}

Let us define strongly/weakly parabolic Higgs bundles. It is reasonable that a general strongly parabolic Higgs bundle should be a cotangent vector of a stable parabolic vector bundle in its moduli space. Recall from (\ref{(3.4)}) that $\sP ar\sE nd(\sE)$ is naturally dual to $\sS\sP ar\sE nd(\sE)(D)$. 
Yokogawa \cite{Yo95} showed:
\[ T^\vee_{\left[\sE\right]}\M_{P,\alpha}^{s}
=(H^1(X,\sP ar\sE nd(\sE)))^\vee
\cong H^0(X,\sS \sP ar \sE nd(\sE)\otimes_{\sO_X}\omega_X(D)).
\]   
So we define strongly parabolic Higgs bundles as follows: 
\bdefinition A \emph{strongly parabolic Higgs bundle} on $X$ with fixed parabolic data $(P,\alpha)$ is a parabolic vector bundle $\sE$ together with a Higgs field $\theta$, 
\[
\theta:\sE\to \sE\otimes_{\sO_X}\omega_X(D)
\] such that $\theta$ is a strongly parabolic map between $\sE$ and $\sE\otimes_{\sO_X}\omega_X(D)$. 
If $\theta$ is merely parabolic, we say that $(\sE,\theta)$ is a \emph{weakly parabolic Higgs bundle}.
\edefinition

\bremark\label{rem:abeliancat}
The category of strongly/weakly parabolic filtered Higgs sheaves is an abelian category with enough injectives which contains the category of strongly/weakly parabolic Higgs bundles as a full subcategory. See \cite[Definition 2.2]{Yo95}. 
\eremark

One can similarly define the stability condition for strongly/weakly parabolic Higgs bundles. A strongly/weakly parabolic Higgs bundle $(\sE,\theta)$ is called $\alpha$-semi-stable (resp. stable) if for all proper sub-Higgs bundles $(F,\theta|_F)\subsetneq (E,\theta)$, one has $\pmu(\sF)\leq \pmu(\sE)$ (resp. $<$). Here by a sub-Higgs bundle, we mean a Higgs bundle $(F,\theta|_{F})$, where $F$ is a subbundle such that $\theta(F)\subset F\otimes\omega_X(D)$. Similar to the vector bundle case, an $\alpha$-stable strongly/weakly parabolic Higgs bundle $(\sE,\theta)$ is simple, i.e., $ParEnd(\sE,\theta)\cong k$.  

As mentioned in the introduction,  Geometric Invariant Theory shows that the $\alpha$-stable objects define moduli spaces $\Higgs^W_{P,\alpha}$ and $\Higgs_{P,\alpha}$ that are normal quasi-projective varieties (see\cite{MS80}, \cite{Yo93C} and \cite{Yo95}). We have 
\beq\label{dimofweakparabolichiggs} \dim(\Higgs^W_{P,\alpha} )= (2g-2+\deg(D))r^2 + 1,\eeq
(see \cite[Theorem 5.2]{Yo95}), and $\Higgs_{P,\alpha}$ is a closed subvariety of $\Higgs^W_{P,\alpha}$ with
\[\dim(\Higgs_{P,\alpha}) =2(g-1)r^2 +2+ \sum_{x\in D}2\dim(G/P_{x})=2\dim(\M_{P,\alpha}),\]
(see \cite[Remark 5.1]{Yo95}).  

For a generic $\alpha$, a bundle (or pair) is $\alpha$-semistable if and only if it is $\alpha$-stable. In these cases, the moduli spaces $\M_{P,\alpha}$, $\Higgs_{P,\alpha}$ and $\Higgs^W_{P,\alpha}$ are nonsingular. 

\vspace{10pt}

\begin{assumption}\label{assm 0}
	In  what  follows,  we will always assume the parabolic weight $\alpha$ is generic in this sense. For simplicity, we will always drop the weight $\alpha$ in the subscripts and abbreviate the parabolic structure $(P,\alpha)$ as $P$.
\end{assumption}

\section{The parabolic Hitchin maps} \label{sec 3}

Weakly parabolic Hitchin maps are defined by Yokogawa \cite[Page 495]{Yo93C}. According to \cite[Theorem 4.6]{Yo93C} and \cite[Remark 5.1]{Yo95},  $\Higgs^W_P$ is a good quotient by an algebraic group $\PGL(V)$ of some $\PGL(V)$-scheme $\mathcal{Q}$. In fact, since we assume the weight $\alpha$ is generic, see Assumption \ref{assm 0}, $\Higgs^W_P$ is a  geometric quotient. On $X_{\sQ}=X\times \sQ$  one has a universal family of stable weakly parabolic Higgs bundles $(\tilde{\sE},\tilde{\theta})$ and a surjection $V\otimes_k \sO_{X_\sQ}\twoheadrightarrow \tilde{\sE}$. Thus the coefficients of the characteristic polynomial of $\tilde{\theta}$ 
\begin{displaymath}
	(a_1(\tilde{\theta}),\cdots, a_n(\tilde{\theta})):=(\tr_{\sO_{X_\sQ}}(\tilde{\theta}),\tr_{\sO_{X_\sQ}}(\wedge^2_{\sO_{X_\sQ}}\tilde{\theta}), \cdots,\wedge^r_{\sO_{X_\sQ}}\tilde{\theta})
\end{displaymath}
determine a section of $\bigoplus_{i=1}^r(\pi_X^*\omega_X(D))^{\otimes i}$ over $X_{\sQ}$. We write $\mathbf{H}^0(X,(\omega_X(D))^{\otimes i})$
for the affine variety underlying $H^0(X,(\omega_X(D))^{\otimes i})$.
Since 
\[
H^{0}(X_\sQ,\bigoplus_{i=1}^r(\pi_X^*\omega_X(D))^{\otimes i}) =\Hom_{\Sch}(\sQ,\prod_{i=1}^r \mathbf{H}^0(X,(\omega_X(D))^{\otimes i})),
\]
the characteristic polynomial of $\tilde{\theta}$ defines a morphism of schemes  
\[
\sQ\to \prod_{i=1}^r \mathbf{H}^0(X,(\omega_X(D))^{\otimes i}).
\]
This map is equivariant under the $\PGL(V)$-action \cite[p.\ 495]{Yo93C} and hence factors through the moduli space $\Higgs^W_P$.  

\bdefinition 
The \emph{Hitchin base space} for the pair $(X,D)$ is 
\[
\sH:=\prod_{i=1}^r \mathbf{H}^0(X,(\omega_X(D))^{\otimes i}),
\] 
and 
\[
h_P^W: \Higgs^W_P\to \sH
\] 
is called the \emph{weakly parabolic Hitchin map}. Notice that $\sH$ does not depends on parabolic structures, so we simply call it the \emph{Hitchin base space}.
\edefinition

Note that  $h_P^W$ is pointwisely defined as $(\sE,\theta)\mapsto (a_1(\theta),\cdots,a_r(\theta))\in \sH$.
Then by the Riemann-Roch theorem, we have
\beq\label{dimofweakparabolichitchinbase} \dim (\sH) =r^2(g-1)+\frac{r(r+1)\deg(D)}{2},\eeq and in general a generic fiber of $h^W_{P}$ has a smaller dimension than $\sH$. 

We shall now define an affine subspace $\sH_P$ of $\sH$ (which, as the notation indicates, depends on $P$) such that $h_{P}^W(\Higgs_P)\subset \sH_P$.
Baraglia and Kamgarpour \cite{BK18} have already determined parabolic Hitchin base spaces for all classical groups\footnote{Their notation for $\sH_P$ is $\sA_{\sG,P}$.}. Moreover when $k=\mathbb{C}$, they show in \cite{BKV18} that $h_{P}$ is surjective by symplectic methods. 
We here do the calculation for $G=\GL_{r}$, not just for completeness, but also because it involves some facts of Young tableaux which we will need later. 
Our proof is simple and direct. In Section \ref{sec bnr}, we will give a proof of surjectivity over general $k$.
\\

\subsection{Intermezzo on partitions}\label{Intermezzo}
A partition of $r$ is  a sequence of integers $n_{1}\geq n_{2}\geq\cdots\geq n_{\sigma}> 0$ with sum $r$. Its conjugate partition is the sequence of integers $\mu_{1}\geq\mu_{2}\geq\cdots\geq\mu_{n_{1}}>0$ (also with sum $r$) given by  
$\mu_{j}=\#\{\ell:n_{\ell}\geq j, 1\leq\ell\leq \sigma\}.$ It is customary to depict this as a Young diagram:
for example for $(n_{1},n_{2},n_{3})=(5,4,2)$,  we have the Young diagram:
\begin{displaymath}
	\large\yng(5,4,2)
\end{displaymath}
We can read the conjugate partition from the diagram: $$(\mu_{1},\mu_{2},\mu_{3},\mu_{4},\mu_{5})=(3,3,2,2,1).$$
Number the boxes as indicated:

\begin{center}
	\small\begin{ytableau}
		1&4&7&9&11\\
		2&5&8&10\\
		3&6
	\end{ytableau}
\end{center}
For each partition of $r$, we assign a level function: $j\rightarrow \gamma_{j}, 1\leq j\leq r$, such that $\gamma_{j}=l$ if and only if $$\sum_{t\leq l-1}\mu_{t}< j\leq\sum_{t\leq l}\mu_{t}.$$
For example, combined with the former numbered Young Tableau, $\gamma_{j}$ is illustrated as follows:
\begin{center}
	\begin{equation}\label{ex:young}
		\small\begin{ytableau}
			1&2&3&4&5\\
			1&2&3&4\\
			1&2
		\end{ytableau}
	\end{equation}	
\end{center}
It is clear that:
\begin{align}\label{comb fact 1}
	&\sum_{j}\gamma_{j}=\sum_{t}t\mu_{t}=\sum_{i}\sum_{j\leq n_{i}}j=\sum_{i}\frac{1}{2}n_{i}(n_{i}+1),\\
	\label{flagmui}&\sum_{i=1}^{\sigma}(n_i)^{2}=\sum_{t=1}^{n_{\sigma}}t^{2}(\mu_{t}-\mu_{t+1})
	=\sum_{t=1}^{n_{\sigma}}(2t-1)\mu_{t}.
\end{align}    

In the following, we reorder the Levi type $\{m^{j}(x)\}_{j=1}^{\sigma_{x}}$ from large to small as $\{n_{j}(x)\}_{j=1}^{\sigma_{x}}$, so that
$n_1(x)\ge n_2(x)\ge \cdots \ge n_{\sigma_x}(x)>0$. This  is a partition of $r$.
\bdefinition The \emph{parabolic Hitchin base} for the parabolic data $P$ is
\[
\mathcal{H}_{P}:=\prod_{j=1}^r\mathbf{H}^{0}\Big(X,\omega_X^{\otimes j}\otimes\sO_X\big(\sum_{x\in D}(j-\gamma_{j}(x))\cdot x\big) \Big)\subset \sH,
\]
where the right-hand side is regarded as an affine space.     
\edefinition

\blemma\label{dim formula} $\dim \mathcal{H}_{P}=\frac{1}{2}\dim \Higgs_P$.
\elemma
\begin{proof} 
	Recall that $\dim\Higgs_P=\dim T^{*}\M_{P}=2\dim \M_{P}$.	By the Riemann-Roch theorem, we have
	\begin{align*}
		\dim\mathcal{H}_{P}&=\sum_{j=1}^{r}\dim H^{0}\Big(X,\omega_X^{\otimes j}\otimes\sO_X\big(\sum\limits_{x\in D}(j-\gamma_j(x))\cdot x\big)\Big)\\
		&=1+r(1-g)+\frac{r(r+1)}{2}(2g-2)+\sum_{j=1}^{r}\sum_{x\in D}(j-\gamma_j(x))\\
		&=1+r^{2}(g-1)+\frac{r(r+1)\deg D}{2}-\sum_{x\in D}\sum_{j=1}^{r}\gamma_j(x)\\
		&=\dim(\M_P)+\frac{1}{2}\sum\limits_{x\in D}\Big(r+\sum\limits_{l=1}^{\sigma_x}m^l(x)^{2}-2\sum\limits_{j=1}^{r}\gamma_j(x)\Big)\\
		&=\frac{1}{2}\dim \Higgs_P.	
	\end{align*}
	The last equality follows from (\ref{comb fact 1}) and (\ref{flagmui}).\qed
\end{proof}

\begin{theorem}\label{image parabolic} 
	For $(\sE,\theta)\in\Higgs_P$, $h_{P}^W(\sE,\theta)\in \sH_P$. That is to say, we have 
	$$
	a_{j}(\theta)\in H^{0}\Big(X,\omega_{X}^{\otimes j}\otimes\sO_X\big(\sum_{x\in D}(j-\gamma_{j}(x))\cdot x\big)\Big).
	$$
\end{theorem}

Without loss of generality, we may assume $D=x$. We denote the characteristic polynomial of $\theta$ as $\lambda^{r}+a_{1}\lambda^{r-1}+\cdots a_{r}$ where $a_{i}=\tr(\wedge^{i}\theta)$. 

We denote the formal local ring at $x$ by $\sO$ with the natural valuation denoted by $v$. We denote its fraction field by $\mathcal{K}$. We fix a local coordinate $t$ in a formal neighborhood of $x$ and choose a local section $\frac{dt}{t}$ to get a trivialization of $\omega_{X}(x)$ near $x$. Then the characteristic polynomial around $x$ becomes 
\[f(t,\lambda):=\lambda^{r}+b_{1}\lambda^{r-1}+\cdots b_{r},\]
where $b_i\in \sO$.   
\bproof
Following the above argument, we only need to show that:    	$v(b_{i})\geq\gamma_{i}, \quad 1\leq i\leq r$.
It amounts to proving the following statement:
\begin{claim}
	Let $\sE$ be a free $\sO$-module of rank $r$ with a filtration $F^{\bullet}$ on $\sE\otimes_{\sO} k$. The sequence $\{n_{i}\}$ is $\{\dim_{k}\sfrac{F^{i-1}}{F^{i}}\}_{1}^{\sigma} $ after reordering so that $n_{1}\geq n_{2}\geq\cdots n_{\sigma}>0$. Then for $\theta\in\End_{\sO}(\sE)$ such that $\theta$ respects $F^{\bullet}$, we have $ v(\tr(\wedge_{\sO}^{i}\theta^{i}))\geq \gamma_{i}$.
\end{claim}

Now we prove the claim. Lift $F^{\bullet}$ to a filtration $\sF^{\bullet}$ on $\sE$. This induces a filtration of $\wedge_{\sO}^{i}\sE$ with associated graded $\sO$-module:
$$\bigoplus_{\delta_{1}+\cdots+\delta_{\sigma}=i}\wedge_{\sO}^{\delta_{1}}(\sF^{0}/\sF^{1})\otimes\cdots\otimes\wedge_{\sO}^{\delta_{\sigma}}(\sF^{\sigma-1}/\sF^{\sigma}).$$
Any $\theta$ as above induces a map in each summand, and the trace of this map has valuation no less than $\min\{\delta_{1},\cdots,\delta_{\sigma}\}$. Since $\tr(\wedge_{\sO}^{i}\theta^{i})$ is the sum of these traces, our claim follows from the intermezzo above. \qed
\eproof

Yokogawa \cite[Corollary 5.12, Corollary 1.6]{Yo93C} showed that $h^{W}_{P}$ is projective and $\Higgs_P\subset\Higgs^{W}_{P}$ is a closed sub-variety. By Theorem \ref{image parabolic}, the image of $\Higgs_P$ under $h^W_P$ is contained in $\sH_{P}\subset \sH$.
We denote this restriction by
$$
h_P:=h_P^W|_{\Higgs_P}:\Higgs_P\to\sH_P,
$$
and refer to it as the  \emph{strongly parabolic Hitchin map}. We conclude that:

\begin{propdef}\label{hproper} 
	The  \emph{strongly parabolic Hitchin map} for the parabolic structure $P$ is the morphism 
	$$
	h_P=h_P^W|_{\Higgs_P}:\Higgs_P\to\sH_P.
	$$
	In particular, this morphism is projective.
\end{propdef}

\subsection{Spectral curves} In the next two sections, we determine the generic fibers of the strongly parabolic Hitchin map. As in \cite{BNR}, we introduce spectral curves to realize the Hitchin fibers as a particular kind of sheaves on spectral curves. We should also point out that the classical Beauville-Narasimhan-Ramanan correspondence also holds in positive characteristics, see for example \cite[Corollary 4.5]{BB07}.

One observes that $\sH$ is also the Hitchin base of $\omega_{X}(D)$-valued Higgs bundles. So for $a\in \sH$, one has the spectral curve  $X_a\subset \P_X(\sO_X\oplus \omega_X(D))$ for $\omega_X(D)$-valued Higgs bundles, cut out by the characteristic polynomial $a\in \sH$. Denote the projection by $\pi:X_a\to X$. One can compute the arithmetic genus as:
$$p_{a}(X_{a})=1-\chi(X,\pi_{*}\sO_{X_{a}})
=1+r^{2}(g-1)+\frac{r(r-1)}{2}\deg(D).
$$

When we work 
in the weakly parabolic case, $X_a$ is smooth for generic $a\in \sH$.  On the other hand, for any $a\in \sH_P$, the spectral curve $X_a$ is singular (except for the Borel case). Yet for a generic $a\in\sH_{P}$, $X_{a}$ is integral, totally ramified at $x\in D$ and smooth elsewhere. Please refer to the Appendix.

\section{Generic Fibers of the Strongly Parabolic Hitchin Map}\label{sec bnr}

In this section, we determine the generic fibers of the strongly parabolic Hitchin map. We will start from a local analysis, and then derive from it the parabolic BNR correspondence as stated in Theorem \ref{main02}. The analysis of the local case is also of its own interest.

\subsection{Local Strongly Parabolic Higgs Bundles}\label{local analysis}
Let us now clarify our local strongly parabolic Higgs bundles:
\bdefinition\label{local data} A \emph{local strongly parabolic Higgs bundle} is a triple $(V,F^\pt,\theta)$ satisfying the following conditions:
\begin{itemize}
	\item[(a)] $V$ is a free $\sO=k[[t]]$-module of rank $r$, endowed with a filtration $F^\pt V$: $$V=V^{0}\supset V^1\supset\cdots\supset V^{\sigma}=t\cdot V,$$
	with $\dim\sfrac{V^{i}}{V^{i+1}}=m_{i+1}$. As before we rearrange $(m_{i})$ as $(n_{i})$ to give a partition of $r$.
	\item[(b)] $\theta:V\rightarrow V$ is a $k[[t]]$-linear morphism and $\theta(V^{i})\subset V^{i+1}$. 
\end{itemize}
We call $(V,F^\pt,\theta)$ a \emph{distinguished local strongly parabolic Higgs bundle} if in addition, the following holds:
\begin{itemize}
	\item [(c)] $f(\lambda,t)$ is the characteristic polynomial of $\theta$ and has a factorization, $$f(\lambda,t)=\prod_{i=1}^{n_{1}}f_{i},$$
	such that each $f_{i}$ is an Eisenstein polynomial with $\deg(f_{i})=\mu_{i}$, where $(\mu_1,\ldots,\mu_{n_1})$ as before is the conjugate partition of the partition $(n_i)$. Besides, if $\deg f_{i}=\deg f_{j}$, then the difference of their constant terms lies in $t\cdot k[[t]]\backslash t^{2}\cdot k[[t]]$.
\end{itemize}
\edefinition

\begin{remark}
	A monic polynomial $g=\lambda^{d}+\sum_{1\leq i\leq d} a_i\lambda^{d-i}\in k[[t]][\lambda]$ is called Eisentein if the coefficients of $g$ satisfy that $v(a_{i})\geq 1$ for $1\leq i \leq d-1$, and $v(a_{d})=1$ with the valuation $v$ on $k[[t]]$ defined by the maximal ideal $(t)$. By \cite[Chapter I, Section 6, Proposition 17]{Ser79}, if $g$ is Eisentein then $\frac{k[[t]][\lambda]}{(g)}$ is a discrete valuation ring (DVR). The last condition in (c) on ``constant terms" is essentially used in Equation \eqref{eq:diff const terms}, to get an inverse of a certain power series of linear morphisms.
\end{remark}

We write $A:=\sO[\lambda]/(f)$, and $A_i:=k[[t]][\lambda]/(f_i(t,\lambda))$. Since each $f_{i}$ is an Eisenstein polynomial, each $A_{i}$ is a DVR and we put $\tilde{A}:=\prod_{i=1}^{n_{1}}A_i$. Then we have a natural injection $A\hookrightarrow \tilde{A}$ and $\tilde{A}$ can be treated as the normalization of $A$.

\btheorem[Local Parabolic BNR Correspondence]
A distinguished local strongly parabolic Higgs bundle $(V,F^\pt,\theta)$ induces a principal $\tilde{A}$-module structure on $V$.
\etheorem
\vspace*{5pt}

\noindent We shall prove this theorem through the following two propositions. From Subsection \ref{Intermezzo}, we know that $\sigma=\mu_{1}$. Here we recall that $\sigma$ is the length of the filtration. Thus $\theta^{\mu_{1}}(v)\in tV$ for any $v\in V$. 

We define $\Ker f_{i}:=\{v\in V\ |\ f_{i}(\theta)(v)=0\}$. Let $w$ be an element of $V$ and $g$ be an element in $\sO$. If $g\cdot w\in \Ker f_{i}$, then $g\cdot f_{i}(\theta)(w)=f_{i}(g\cdot w)=0$ and as a result, $w\in \Ker(f_{i})$. Thus $V/\Ker f_{i}$ is torsion-free, which means that $\Ker f_{i}$ is a direct summand. 

\begin{proposition}\label{first step direct summand}
	Let $(V,F^\pt,\theta)$ be a distinguished local strongly parabolic Higgs bundle, then the image of 
	$$\Ker f_{1}\rightarrow V\rightarrow {V/\Ker f_{i}}$$
	is $\Ker\bar{f}_{1}$, for $1< i\leq n_{1}$.
	Here $\bar{f}_{1}$ is the induced map of $f_{1}$ on $V/\Ker f_{i}$.	In particular, $\Ker f_{1}\oplus \Ker f_{i}$ is a direct summand of $V$.
\end{proposition}
\begin{proof}
	We denote the natural quotient map $V\rightarrow V/\Ker f_{i}$ by $q_{i}$. Then
	$$q_{i}^{-1}(\Ker\bar{f}_{1})=\{v\in V\ |\ f_{1}(v)\in \Ker f_{i}\}.$$
	For simplicity we write $f_{1}$ as $\lambda^{\mu_{1}}+\alpha_{1}$, where $\alpha_{1}\in tk[[t]]\backslash t^{2}k[[t]]$ by the genericity condition. We denote $f_{1}(v)=w\in \Ker f_{i}$. By definition, to show $q_{i}(\Ker f_{1})=\Ker\bar{f}_{1}$, it suffices to show $$\text{there exists } w^{'}\in \Ker f_{i},\,  \text{ such that } f_{1}(w^{'})=w.$$
	This amounts to solving the following linear equations:
	
	\[\left\{ \begin{aligned} (\theta^{\mu_{1}}+\alpha_{1})\ w^{'} & =w\\
		(\theta^{\mu_{i}}+\alpha_{i})\ w^{'}& = 0
	\end{aligned}\right.\]
	i.e.,	$(-\alpha_{i}\ \theta^{\mu_{1}-\mu_{i}}+\alpha_{1})\ w^{'}=w$.
	
	It is easy to see that $\theta$-action on $V$ is continuous with respect to $t$-adic topology on $V$, and thus $V$ can be treated as $ k[[t]][[\lambda]]$-module. 	We rewrite:
	\begin{equation}\label{eq:diff const terms}
		(-\alpha_{i}\theta^{\mu_{1}-\mu_{i}}+\alpha_{1})=t\phi_{t}(\theta),
	\end{equation}
	then $\phi_{t}(\theta)^{-1}$ is a well-defined map on $V$, since we assume that if $\deg f_1=\deg f_{i}$ then the constant term of $\alpha_{i}-\alpha_{1}$ is not in $t^{2}k[[t]]$.

	Notice that $f_{1}(v)=w$ implies $\theta^{\mu_{1}}(v)\equiv w (\text{mod }  t)$. Since $\theta^{\mu_{1}}v\in tV$, we know that $w\in tV$, then we can find a (unique) $$w^{'}=\phi_{t}^{-1}(w/t),$$ such that $q_{i}(v-w^{'})=q_{i}(v)$ and $v-w^{'}\in \Ker f_{1}$. Thus $q_{i}:\Ker f_{1}\rightarrow \Ker\bar{f}_{1}$ is surjective.
	
	Since $\Ker\bar{f}_{1}$ is a direct summand of $V/\ker f_{i}$, 
	$\Ker f_{1}\oplus \Ker f_{i}$ is a direct summand of $V$. \qed
\end{proof}
\begin{proposition}\label{decomposition we want}
	Let $(V,F^\pt,\theta)$ be a distinguished local strongly parabolic Higgs bundle, then we have the following decomposition:
	$$V\cong \bigoplus_{i=1,\ldots,n_{1}}\Ker f_{i}.$$
	That is to say $V$ is a principal $\tilde{A}$-module. Moreover, for each $V^{j}$, $0\leq j\leq \sigma$, we still have decompositions $$V^j\cong \bigoplus_{i=1,\ldots,n_{1}}( V^j\cap\Ker f_{i}).$$
\end{proposition}

\begin{remark}
	It is obvious that we \textit{cannot} in general lift an $A$-module structure to an $\tilde{A}$-module structure. For example, as an $A$ module, $A$ itself would not have a $\tilde{A}$ module structure in general. The reason lies in that over a principal $A$-module, we do not have a filtration $F^\pt$ of Levi type $(m^1, m^2,\ldots,m^{\sigma})$, and a local Higgs field $\theta$ that strongly preserves it. This proposition actually shows the effect of the parabolic condition on the local structure of Higgs bundles.
\end{remark}
\bproof
We prove this by inductions both on the rank of $V$ and the number of irreducible factors of the characteristic polynomial of $\theta$. From Proposition \ref{first step direct summand}, we know that $\Ker f_{1}\oplus \Ker f_{i}$ is direct summand of $V$ for $2\leq i\leq n_{1}$.

Let us now consider the quotient map:
$$q_{1}:V\rightarrow V/\Ker f_{1}.$$
Since $\Ker f_1\oplus\Ker f_i$ is a direct summand of $V$ by Proposition \ref{first step direct summand}, $q_{1}(\Ker f_{i})$ is a direct summand and is contained in $\Ker\bar{f_{i}}\subset V/\Ker f_{1}$.
Because $\Ker f_{i}\cap \Ker f_{1}=\{0\}$, $q_{1}$ is injective when restricted to $\Ker f_{i}$. By passing to $V\otimes_{\sO} K$, we have the obvious decomposition:
$$V\otimes_{\sO} K=\bigoplus_{i=1}^{n_{1}}\Ker f_{i}\otimes_{\sO}K.$$
We know that $\text{rk}(\Ker f_{i})=\text{rk}(\Ker\bar{f}_{i})$, 
then:
$$q_{1}(\Ker f_{i})=\Ker\bar{f}_{i}.$$
Thus we only need to prove that:
$$V/\Ker f_{1}=\bigoplus_{i=2}^{n_{1}}\Ker\bar{f}_{i}.$$

Recall that $\theta$ acts on $V/\Ker f_{1}$ with characteristic polynomial $\prod_{i=2}^{\sigma}f_{i}$. The filtration on $V$ actually induces a filtration on $V/\Ker f_{1}$.\footnote{Because $\Ker f_{1}\cap V_{j}$ is a direct summand of $V_{j}$.} To use induction, we only need to show that the length of this filtration is $\mu_{2}$. This follows from that $\Ker f_{1}$ is rank one module over $A_{1}$ which is a DVR. Then by induction, we have decomposition on $V/\Ker f_{1}$, i.e:
\begin{equation}
	V/\Ker f_{1}\cong \bigoplus_{i=2,\ldots,n_{1}}\Ker \bar{f}_i.
\end{equation}
As $q_{1}:\Ker f_{i}\rightarrow \Ker\bar{f}_i$ is surjective, we have the decomposition.

Now we prove the last assertion. Indeed, consider the following filtration on $V^{j}$:
\[
V^{j}\supset V^{j+1}\supset\cdots\supset tV\supset tV^{1}\supset\cdots\supset tV^{j}.
\]
and $\theta:V^{j}\rightarrow V^{j}$. It also satisfies conditions (a),(b),(c) in Definition \ref{local data}, and thus $V^{j}$ has the decomposition:
\[
V^{j}=\oplus_{i=1,\ldots, n_{1}}\Ker f_{i}|_{V^{j}}=\oplus_{i=1,\ldots, n_{1}} (V^{j}\cap\Ker f_{i}).
\]
We finish the proof.\qed
\eproof

To capture the parabolic structure, we define the ``Young diagram filtration" for a principal {$\tilde A$}-module.

\begin{propdef}\label{propdef:young filtration} The filtration $F^\pt$ on $V$ can be canonically recovered by the principal $\tilde A$-module structure on $V$ which we call the Young diagram filtration.
\end{propdef}

\begin{proof}
	
	Recall that $V\cong \bigoplus_{i=1,\ldots,n_{1}}\Ker f_{i}$ as in Proposition \ref{decomposition we want}, and for each $i$, $\Ker f_{i}$ is a principal module of $A_i=k[[t]][\lambda]/(f_i(t,\lambda))$, where $f_i(\lambda, t)$ is an Eisenstein polynomial. So there is a filtration on $\Ker f_i$: $$\Ker f_i\supset \lambda \Ker f_i \supset \cdots \supset \lambda^{\mu_i}\Ker f_i,$$ such that each graded piece has dimension 1 as a $k$ vector space. Alternatively, if we restrict the filtration $F^\pt$ on $V$ to $\Ker f_i$, we get another filtration: $$\Ker f_i\supset V^1\cap \Ker f_i \supset \cdots \supset tV\cap \Ker f_i=\lambda^{\mu_i}\Ker f_i.$$ Since $\theta(V^i)=\lambda V^i\subset V^{i+1}$, we see that each graded piece of this filtration has dimension $1$ or $0$. So if we eliminate the abundant terms in the second filtration, these two filtrations on $\Ker f_i$ should be the same and hence $V^j\cap \Ker f_i=\lambda^{a_{ij}}\Ker f_i$ for some $a_{ij}\in \N$.
	
	Applying Proposition \ref{decomposition we want}, we know that the filtration $F^\pt$ on $V$ can be written as: $$V\cong \bigoplus_{i=1,\ldots,n_{1}}\Ker f_{i}\supset \bigoplus_{i=1,\ldots,n_{1}}\lambda^{a_{i1}}\Ker f_{i}\supset \bigoplus_{i=1,\ldots,n_{1}}\lambda^{a_{i2}}\Ker f_{i}\supset \cdots \supset \bigoplus_{i=1,\ldots,n_{1}}\lambda^{a_{i\sigma}}\Ker f_{i}.$$
	In other words, the filtration $F^\pt$ is totally determined by the $\tilde A$-module structure (multiplied by $\lambda$ and direct sum decomposition) and the numbers $\{a_{ij}\}$.
	
	We claim that those $\{a_{ij}\}$ can be recovered by the Levi type $(m^1,\ldots, m^\sigma)$ of $F^\pt$. If so, then we see that the filtration $F^\pt$ on $V$ can be uniquely recovered by the $\tilde{A}$-module structure on $V$ provided with the fixed Levi type $(m^1,\ldots, m^\sigma)$.
	
	We now prove the claim. We write $b_{ij}:=a_{ij}-a_{i,j-1}$. Then $\{b_{ij}\}$ satisfy the following equations:
	\[	\sum_{i}b_{ij}=m^i;\;\ 
	\sum_{j}b_{ij}=\mu_i;\;\ 
	b_{ij}\in \{0,1\} .
	\]

	To prove the claim, it suffices to prove that there exists a unique $B=(b_{ij})$ with entry $0$ or $1$ such that $$(1, \ldots, 1)B= (m^1, \ldots, m^{\sigma})\ \ \text{and}\ \ (1, \ldots, 1)B^{\text{T}}= (\mu_1, \ldots, \mu_{n_1}).$$
	
	Notice that $(n_1 \geq  \cdots \geq n_{\sigma})$ and $(\mu_1\geq \cdots \geq \mu_{n_1})$ are conjugate partitions, then there exists a unique matrix $\tilde{B}$ by the Young diagram such that
	$\tilde{B}=(\tilde{b}_{ij})$ has entries $0$ or $1$ and $$(1, \ldots, 1)\tilde{B}= (n_1, \cdots, n_{\sigma})\ \ \text{and}\ \ (1, \ldots, 1)\tilde{B}^{\text{T}}= (\mu_1, \ldots, \mu_{n_1}).$$
	
	It indicates that the matrix $B$ exists uniquely since $\{m^i\}$ is a rearrangement of $\{n_i\}$.\qed
	
\end{proof}
\bremark In particular, if the parabolic structure is Borel, the spectral curve is already smooth. In our local case, we have $A=\tilde A$. In a distinguished local strongly parabolic Higgs bundle $(V,F^\pt,\theta)$, the filtration $F^\pt$ is a filtration of $A$-modules and indeed a filtration of $\tilde A$-module since it is preserved by $\theta$. Distinguishness yields the isomorphism $V\cong \tilde A$ and the strong parabolic condition of $\theta$ implies $F^\pt$ is finer than the $(\lambda)$-adic filtration. Then the Levi type of the grading forces the isomrophism $(V,F^\pt)\cong (\tilde A,(\lambda)\text{-adic})$, i.e., the filtration $F^\pt$ is completely determined by the $\tilde A$-module structure on $V$.
\eremark

In the remaining part of the section, for simplicity, we assume $D=x$. Our first goal is to normalize the singular spectral curve $X_{a}$ and analyze the local properties of its normalization.

\subsection{Normalization of spectral curves}
We denote $N:\tilde{X}_{a}\to X_a$ as the normalization of $X_{a}$ and denote by $\tilde{\pi}$ the composition map: 
$$\tilde{\pi}: \tilde{X}_{a}\xrightarrow{N} X_{a}\xrightarrow{\pi} X.$$
As in Theorem \ref{image parabolic}, $f\in \sO[\lambda]\cong k[[t]][\lambda]$ defines the spectral curve locally, so that the formal completion of the local ring of $X_a$ at $x$ is $A:=\sO[\lambda]/(f)$. Notice that $\Spec (A)$ and $X_a-\{\pi^{-1}(x)\}$ form an fpqc covering of $X_a$. Since $X_a-\{\pi^{-1}(x)\}$ is smooth, we only need to construct the normalization of $\Spec(A)$. 

\begin{figure}[h]
	\centering
	\includegraphics[width=\textwidth]{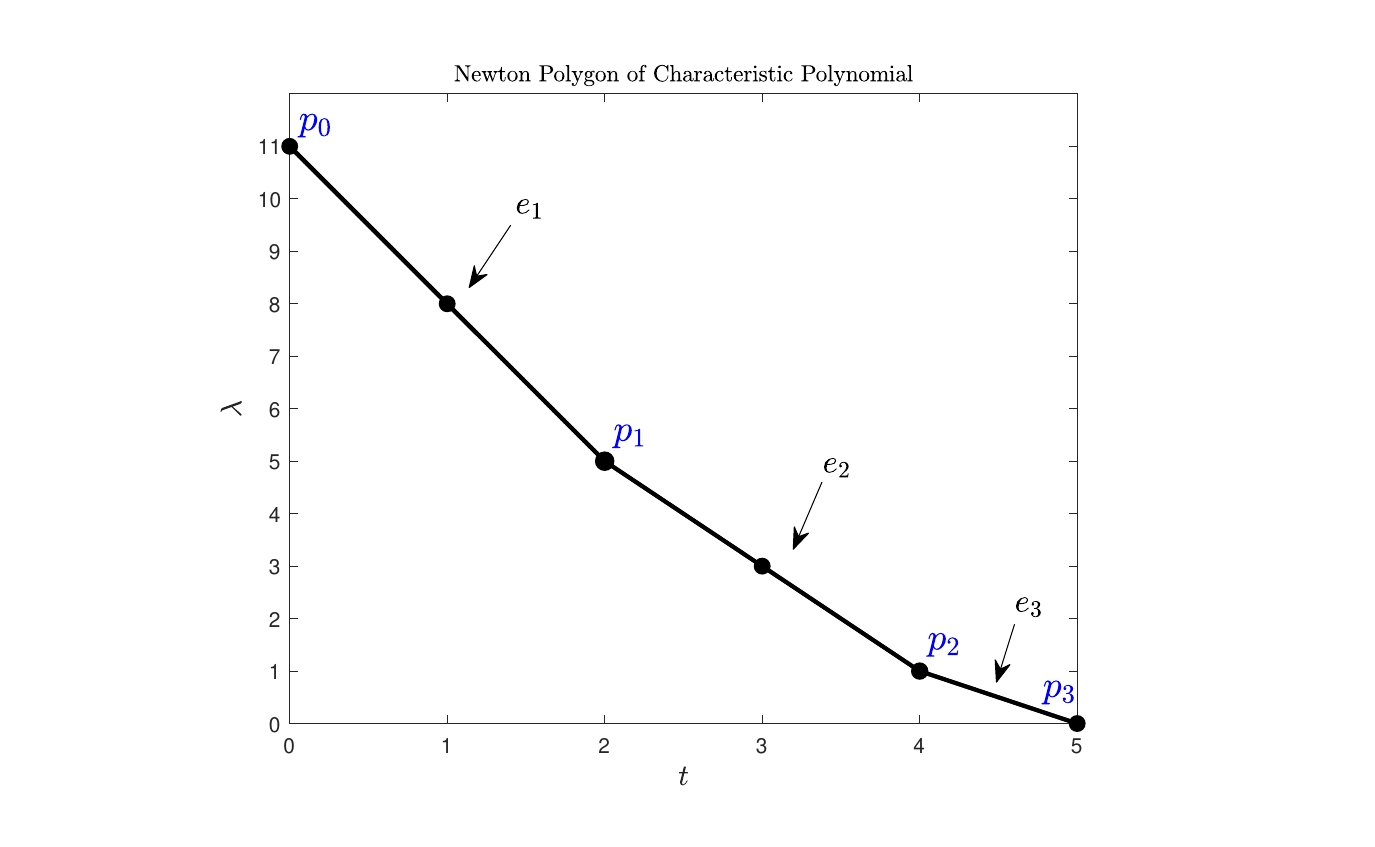}
	\caption{Newton Polygon of Characteristic Polynomial}
	\label{Huawei}
\end{figure}

From Theorem \ref{image parabolic} in Subsection \ref{Intermezzo}, there is an open subset $U_{\text{val}}\subset \sH_P$ such that for all $a\in U_{\text{val}}$, the coefficient $b_i\in \sO$ has the minimal valuation $\gamma_i$ determined by the parabolic data. From now on, we assume $a\in U_{\text{val}}$ and denote the Newton polygon of the corresponding characteristic polynomial $f$ by $\Gamma$. To give a better picture, we put  Figure \ref{Huawei} above which is the Newton Polygon of a characteristic polynomial corresponding to Example \ref{ex:young} in Subsection \ref{Intermezzo}.

We define $C=\Gamma+\R_{\ge 0}^2$, so that $\Gamma$ determines $C$ which is a closed convex subset of $\R_{\ge 0}^2$. Let $p_0, p_1, \dots ,p_s$, be the ``singular" points of $\partial C$:
points where it has an angle $<\pi$ (so that $p_s$ lies on the $x$-axis).  The standard theory of toric modifications was developed in the 1970's and is due to several people. For the construction we refer the readers to the introduction paper \cite{Oka09} and references therein. It assigns to $\Gamma$ a  toric modification $\pi: T_\Gamma\to \A^2$ of $\A^2$, where $T_\Gamma$ is a normal variety.

The morphism $\pi$ is proper and is an isomorphism over $\A^2\backslash \{(0,0)\}$. The exceptional locus $\pi^{-1}(0,0)$ is the union of irreducible components $\{D_e\}$, where $e$ runs over the edges of $\Gamma$. 
For every edge $e$ of $\Gamma$, denote by $f_e$ the subsum of $f$ over $e\cap \Z^2$.  

\begin{assumption}\label{generic on boud}
	We impose the genericity condition that all these roots of $f_{e}$
	are nonzero and pairwise distinct for all $e\in\partial \Gamma$. This is the concept ``non-degeneracy" in \cite{Oka09}. 
\end{assumption} 

Under this assumption, according to \cite[Theorem 22]{Oka09}, the strict transform $\hat Z(f)$ of $Z(f)\subset \A^{2}$ in  $T_\Gamma$ is the normalization of $Z(f)$, and meets  $D_e$ transversally in a set that can be effectively indexed by the connected components of $e\backslash \Z^2$. In particular, $Z(f)$ has as many branches at the origin as connected components of $\Gamma\backslash \mathbb{Z}^2$.

In our case, the slope of $e$ is $-\mu_{e}$ for some $\mu_{e}\in\{\mu_{1},\ldots,\mu_{n_1}\}$, and each branch of $Z(f)$ whose strict transform meets $D_e$ transversally is totally ramified over $x$ with ramification degree $\mu_{e}$. By \cite[Chapter I, Section 6, Proposition 18]{Ser79}, the formal local ring of such a point is a DVR totally ramified over $k[[t]]$ since $k$ is algebraically closed and the characteristic polynomials of its uniformizers are Eisenstein polynomials of degree $\mu_{e}$. In particular, under our notations here, we can choose $\lambda$ as the uniformizer, and thus $\lambda$ has to satisfy an Eisentein polynomial for each ramification point over $x$. To conclude:
\begin{proposition}\label{decomposition of char}
	Under Assumption \ref{generic on boud}, $f$ decomposes in $k[[t,\lambda]]$ into a product of Eisenstein polynomials $f=\prod_{i=1}^{n_1} f_i$. Exactly  $\#\{i|\mu_{i}=\mu_{e}\}$ of them are of degree $\mu_{e}$ (but the difference of two such have their constant terms not divisible by $t^2$, since the roots of $f_e$ are pairwise different), and $\prod_{i=1}^{n_{1}} k[[t]][\lambda]/(f_i)$ is the normalization of $k[[t]][\lambda]/(f)$. 
\end{proposition}

This is a stronger conclusion than that in \cite[Chapter 2, Proposition 6.4]{Neu99}, because of our Assumption \ref{generic on boud} which is an extra condition on the separability of Newton polygons. Now we can make the genericity condition precise. 
\bdefinition\label{def:genericity condition}
We define a subvariety  $\sU$ of $U_{\text{val}}$ such that for any $a\in\sU$, satisfying the following properties:
\begin{enumerate}
	\item[(a)] The spectral curve $X_{a}$ is integral and has a unique singularity $\pi^{-1}(x)$.
	\item[(b)] The Newton polygon of the characteristic polynomial associated to $a$ satisfies Assumption \ref{generic on boud}.
\end{enumerate}
By Lemma \ref{generic integrality} in the Appendix, the first condition is an open condition, and the separability condition in Assumption \ref{generic on boud} is also an open condition, and thus $\sU$ is an open subvariety.
\edefinition

\begin{corollary}\label{ramification} 
	For all $a\in \sU$, there are $n_{1}$ (the length of conjugate partition) points in $\tilde{X}_{a}$ over $x\in X$ with ramification degrees as $(\mu_{1},\mu_{2},\ldots,\mu_{n_{1}})$ (the conjugate partition of $(n_{1},\ldots,n_{\sigma})$). The geometric genus of $\tilde{X}_a$ is $$ g(\tilde{X}_a)=r^2(g-1)+1+\dim(G/P_x).$$
\end{corollary}
\begin{proof}
	By Proposition \ref{decomposition of char}, the formal local ring of $\tilde{X}_{a}$ can be written as $\prod_{i=1}^{n_{1}}k[[t]][\lambda]/(f_i)$. The ramification degrees are due to the degrees of Eisenstein polynomials defining strict transforms of local branches. The geometric genus $g(\tilde{X}_a)$ then follows from the ramification degree formula in \cite[Chapter 5, Example 3.9.2]{Hart77}.\qed
\end{proof}

\subsection{Parabolic BNR correspondence}

This subsection is devoted to building up the parabolic BNR correspondence.

\begin{theorem}[Parabolic BNR Correspondence for $\GL_{r}$]\label{parabolic BNR}For any $a\in \sU$, the direct image $\tilde{\pi}_*$ of $\tilde{\pi}:\tilde{X}_a\to X$ induces a functorial isomorphism: 
	$$\begin{array}{rccc}\tilde{\pi}_*:&\{\text{degree } \delta \text{ line bundles over } \tilde{X}_a\}&\to& \left\{\begin{array}{c}\text{Strongly parabolic Higgs bundle } (\mathcal{E},\theta)\\
			\text{ with }\\ h_{P}((\mathcal{E},\theta))=a,\ \deg(E)=d 
		\end{array} \right\}\\
		&&&\\
		& L&\mapsto &\begin{array}{c}
			\tilde\pi_*L\text{ with the Young diagram filtration}\\\text{ as in Proposition-Definition} \; \ref{propdef:young filtration}
		\end{array}
	\end{array}$$
	where $\delta=(r^2-r)(g-1)+\dim(G/P_x)+d$. In particular, the fiber $h_{P}^{-1}(a)$ is isomorphic to $\Pic^{\delta}(\tilde{X}_{a})$.
\end{theorem} 

\begin{proof}
	Assume $L$ is a line bundle on $\tilde{X}_a$. Then the direct image $\tilde\pi_* L$ has a $\tilde\pi_*\sO_{\tilde X_a}$-module structure which makes $\tilde\pi_*L$ an $\omega_X(x)$-valued Higgs bundle as in the classical BNR correspondence. Now by Proposition-Definition \ref{propdef:young filtration}, $\tilde\pi_* L$ has a canonical Young diagram filtration with parabolic type $P_x$. To be more precise, we put $\tilde{\pi}^{-1}(x)=\{y_{1},\ldots,y_{n_{1}}\}$ with ramification index $\mu_{1},\ldots, \mu_{n_{1}}$ respectively. Then we have the filtration on 
	\[
	L\supset L(-\sum_{i=1,\ldots,n_{1}}a_{i1}y_{i})\supset L(-\sum_{i=1,\ldots,n_{1}}a_{i2}y_{i})\supset\cdots L(-\sum_{i=1,\ldots,n_{1}}a_{in_{1}}y_{i})
	\]
	with $a_{ij}$ as in Proposition-Definition \ref{propdef:young filtration}.

	In the next, we show that given a strongly parabolic Higgs bundle $(\sE,\theta)=(E,\theta,P_x)$, there is a unique line bundle $L$ on $\tilde{X}_{a}$ such that $\tilde{\pi}_*L\cong E$. To prove this, we show that the parabolic data endow $E$ with a $\tilde\pi_*\sO_{\tilde{X}_a}$-module structure. 
	
	Recall that $\tilde{\pi}$ decomposes as the composition of two finite morphisms $N:\tilde{X}_a\to X_a$ and $\pi:X_a\to X$. By the classical BNR correspondence, we can realize the underlying Higgs bundle $(E,\theta)$ of $(\sE,\theta)$ as the direct image of a torsion free rank one sheaf $V$ over the singular spectral curve $X_a$. The parabolic structure $P_x$ endows $V$ with a filtration on $V_{\pi^{-1}(x)}$.
	
	We show that $(V,P_x,\theta)$ induces a $N_*\sO_{\tilde{X}_a}$-module structure on $V$. Since the normalization map is finite and isomorphic over $X_a-\pi^{-1}(x)$, we reduce by fpqc descent to consider the problem in a complete neighborhood of $x$. That is to say, we now show that the $\hat{\sO}_{X_a,\pi^{-1}x}$-module $\hat{V}_{X_a,\pi^{-1}x}$ has a principal $\hat{N_*\sO_{\tilde{X}_a}|_{\pi^{-1}x}}$-structure. Since $a\in \sU$, by Proposition \ref{decomposition of char}, $(\hat{V}_{X_a,\pi^{-1}x}, P_x, \theta)$ verifies conditions (a), (b), (c) in Definition \ref{local data}, and thus it is a distinguished local strongly parabolic Higgs bundle. We can apply our Proposition \ref{decomposition we want} and give the principal $\hat{N_*\sO_{\tilde{X}_a}|_{\pi^{-1}x}}$-module structure on $\hat{V}_{X_a,\pi^{-1}x}$. Then $V$ becomes a principal $N_*\sO_{\tilde{X}_a}$-module and must be a direct image of a line bundle $L$.

	The degree $\delta$ can be calculated using the Riemann-Roch theorem, as $g(\tilde{X}_a)=r^2(g-1)+1+\dim(G/P_x)$ in Corollary \ref{ramification}.\qed
\end{proof}
\bremark From Proposition-Definition \ref{propdef:young filtration}, the tame singularity of the spectral curve guarantees that there is only one parabolic structure of fixed type which is compatible with the Higgs field, even if the Higgs field is nilpotent at the points of $D$.
\eremark

\bremark
Scheinost and Schottenloher \cite{SS95} proved a similar result over $\mathbb{C}$ by uniquely extending the eigen line bundle on $X_{a}-\pi^{-1}(D)$ to $\tilde{X}_{a}$. This extension is announced there. We use a different strategy here which is similar to that in \cite{BNR} to prove the correspondence.
\eremark

\subsection{Regularity of (Parabolic) Higgs Fields}\label{subsec:regularity}

In this subsection, we want to explain the importance of the regularity of (parabolic) Higgs fields and its interaction with the geometry of spectral curves.

Let us start without parabolic data and we denote $\Higgs$ as the moduli stack of Higgs bundles. Donagi and Gaitsgory\cite{DG02} and Ngo \cite{Ngo10} prove that there exists an open substack $\textbf{Higgs}^{\text{reg}}$ parametrizing Higgs bundles $(E,\theta)$ with $\theta$ regular everywhere. Moreover, they show that there exists an open subvariety $W$ of the Hitchin base $\sH$ such that $\textbf{Higgs}^{\text{reg}}$ coincides with $\Higgs$ over $W$ and for any $a\in W$, the spectral curve is smooth. This illustrates that the geometry of generic fibers is closely related with the regularity of Higgs fields. 

Now let us return to the parabolic case.
Let $X_a\xrightarrow{\pi} X$ be an integral spectral curve with characteristic polynomial $a\in \sH_P$ and $L$ be a torsion free sheaf of generic rank $1$ on $X_a$ such that $\pi_{*}L$ is an element in the fiber $h_{P}^{-1}(a)$.

We first consider the full flag case.
Locally near $x$, $\pi_*\sO_{X_a}$ can be written  as $\frac{\sO[\lambda]}{f(\lambda)}$ with $f(\lambda)=\lambda^n+a_1\lambda^{n-1}+\cdots
+a_n$ and $v(a_i)=1$, where $v$ is the valuation at $x$ as in Subsection \ref{Intermezzo}. Thus $X_a$ is smooth and the $\pi_*\sO_{X_a}$-module structure on $\pi_*L$ induces a Higgs field $\theta:L\to L\otimes\omega_X(x)$. More concretely, we can choose the trivialization such that $\pi_*L$ and $\pi_*L\otimes \omega_{X}(x)$ are all isomorphic to $\frac{\sO[\lambda]}{f(\lambda)}$ with $\theta$ acting on $\frac{\sO[\lambda]}{f(\lambda)}$ by left multiplication of $\lambda$. Then the matrix of $\lambda$ is a regular nilpotent matrix at $x$ and it preserves the filtration given by the ideal $(\lambda)$. 

However, if the parabolic data $P_x$ is not Borel, $\theta_x$ lies in the nilpotent radical of $P_x$, which is never a regular nilpotent element. In this case, if we assume the characteristic polynomial $a$ lies in our open subset $\sU\subset \sH_P$, then the singularity of $X_a$ is concentrated at $x$ and the local equation of $\pi_*\sO_{X_a}$ is a product of Eisenstein polynomials via Proposition \ref{decomposition of char}. 
Thanks to Theorem \ref{parabolic BNR}, the parabolic structure on $\pi_{*}L$ gives it a principal $\tilde\pi_*\sO_{\tilde X_a}$-module structure and realizes $L$ as a line bundle on $\tilde X_a$.

Applying the decomposition in Proposition \ref{decomposition we want} to $\pi_{*}L=(\sE,\theta)$ in the formal neighbourhood of $x$ implies that $\theta|_x$ lies in the Levi subalgebra $\mathfrak{gl}_{n_1}\times\cdots\times \mathfrak{gl}_{n_{\sigma}}$
and in each factor $\mathfrak{gl}_{n_{i}}$, $\theta_i$ is a regular nilpotent element. To be more precise:

\begin{corollary}\label{integral conjugation}
	For $a\in \sU$ (see Definition \ref{def:genericity condition}) and any $(\sE,\theta)$ in the fiber $ h^{-1}(a)$, the Jordan blocks of $\theta\mod t$ are of size $(\mu_{1},\mu_{2},\ldots,\mu_{n_1})$.
\end{corollary}

Actually, they are the so-called Richardson elements. We refer the readers to \cite{Bau06} for more details.

\begin{remark}
	In \cite[Section 4, Corollary 1]{KL88}, Kazhdan and Lusztig showed that $\theta|_{\Spec(\sK)}$ is conjugate to an element in $GL(\sO)$ whose reduction modulo $t\sO$ is regular nilpotent when the base field is $\mathbb{C}$. Corollary \ref{integral conjugation} is about conjugation over $\sO$, and thus is different.
\end{remark}

\subsection{The $\SL_{r}$ Case}
If we replace $\GL_{r}$ by $\SL_{r}$, we also have coarse moduli spaces $\M_{P,\alpha}^{\circ}$, $\Higgs_{P,\alpha}^{\circ}$. And the parabolic Hitchin space is:
\[
\sH_{P}^{\circ}:=\prod_{j=2}^r\mathbf{H}^{0}\left(X,\omega_X^{\otimes j}\otimes\sO_X(\sum_{x\in D}(j-\gamma_{j}(x))\cdot x)
\right).
\]
We use ``$\circ$" to emphasize the tracelessness. We denote the corresponding strongly parabolic Hitchin map as $h_{P,\alpha}^{\circ}$. Considering the following commutative diagram:
\[
\begin{tikzcd}
	\Higgs^{\circ}_{P,\alpha}\arrow[d]\arrow[r,"{h^{\circ}_{P,\alpha}}"]& \mathcal{H}^{\circ}_{P}\arrow[d]\\
	\Higgs_{P,\alpha}\arrow[r,"{{h_{P,\alpha}}}"]& \mathcal{H}_{P}
\end{tikzcd},
\]
it follows that $h^{\circ}_{P,\alpha}$ is proper. Then by the parabolic BNR correspondence in Theorem \ref{parabolic BNR}, a generic fiber of $h^{\circ}_{P,\alpha}$ is the Prym variety of $\Pic(\tilde{X}_{a})$. Then by a dimension argument and properness, $h^{\circ}_{P,\alpha}$ is surjective.

\section{Generic fiber of weakly parabolic Hitchin maps}\label{section5}

In this section, we give a geometric description of generic fibers of the weakly parabolic Hitchin map (which is called the parabolic Hitchin map in \cite{LM10}). Logares and Martens \cite{LM10} have calculated the number of connected components of generic fibers. The reader is recommended to read \cite[Proposition 2.2]{LM10} if she or he wants to know the concrete number. We here treat it in a more geometric manner for a special case, i.e., the spectral curve is smooth and unramified over the marked point (without loss of generality, we assume that there is only one marked point). In this setting, the number of connected components is equal to the cardinality of the Coxeter group associated to the parabolic subalgebra.

In what follows, for convenience, we fix $a\in \sH$, such that $X_a$ is smooth and unramified over $x$. By Lemma \ref{generic integrality} in the Appendix, such $a\in \sH$ forms an open dense subvariety. To simplify the notation, we omit $\delta$ and use $\Pic(X_{a})$ to denote some connected component of its Picard variety. 


Choosing a marked point $q\in X_{a}$, we have an embedding $\tau:X_{a}\rightarrow \Pic(X_{a})$. We denote by $\ssP$ the pull back of a Poincar\'e line bundle over $\Pic(X_{a})\times X_{a}$. Now let us consider the following projection:
\begin{equation}\label{Univ Higgs from spectra} \Pic(X_a)\times X_a \xrightarrow{\id\times \pi} \Pic(X_a)\times X.
\end{equation}
We put $V:=(\id\times \pi)_{*}\ssP$, which is a rank $r$ vector bundle over $\Pic(X_{a})\times X$. Thus the $(\id\times\pi)_*\sO_{\Pic(X_a)\times X_a}$-module structure induces a Higgs field $$\theta_{\Pic}:V\to V\otimes_{\sO_X}\pi_X^*\omega_X(x).$$ And $(V,\theta_{\Pic})$ can be viewed as the universal family of Higgs bundles on $X$ with characteristic polynomial $a$. For simplicity, we use $V|_{x}$ (resp. $\theta_{\Pic}|_{x}$) to denote the restriction of $V$ to $\Pic(X_{a})\times \{x\}$. For our later purpose, we denote by $\ssA ut(V|_x)$ the group scheme of local automorphisms of vector bundle $V|_x$.

\begin{proposition}\label{Torus over Pic}
	We have a smooth commutative group subscheme $\gT$ of $\ssA ut(V|_x)$ over $\Pic(X_{a})$ acting on $V|_{x}$, such that for any point $z\in \Pic(X_a)$, the fiber $\gT(z)$ is the centralizer of $\theta_{\Pic}|_x$ at $z$.
\end{proposition}  

\begin{proof}
	Restricting $\theta_{\Pic}$ to $\Pic(X_{a})\times \{x\}$ gives $$\theta_{\Pic}|_x:V|_x\to V\otimes_{\sO_X}\pi_x^*\omega_X(x)|_x\cong V|_x.$$
	Since $X_{a}$ is smooth and unramified over $x$, $\theta_{\Pic}|_x$ is regular semisimple everywhere on $\Pic(X_a)$.
	
	Then we consider the centralizer of $\theta_{\Pic}|_x:V|_x\to V|_x$ in $\ssA ut(V|_x)$ over $\Pic(X_{a})\times \{x\}$. This gives us a closed group scheme $\gT$ over $\Pic(X_{a})$, and it is commutative since fiber-wise it is a centralizer of a regular semisimple element in $\mathfrak{g}$. In fact, fiberwise it is a maximal torus. To show the smoothness, we use the trick from Humphreys \cite[Section 10.6]{Hum75} since we deal with $\GL$-case. Notice that $\ssA ut(V|_x)$ is locally a principal open subvariety of $\ssE nd(V|_{x})$, i.e., the bundle of endomorphisms of $V|_{x}$. The centralizer of $\theta$ in $\ssE nd(V|_{x})$ is a vector bundle of rank $r$ (centralizer in Lie algebra is a linear subspace). Thus as an open subvariety of this vector bundle, $\gT$ is smooth.
	
	Besides, there is a natural action of $\gT$ on $V|_{x}$ since it is a subgroup scheme of $\ssA ut(V|_x)$.\qed
\end{proof}

In the following we construct a flag bundle $\mathfrak{Fl}$ on $\Pic(X_a)$ that classifies all the possible filtrations at $x$. Fiber-wise this is isomorphic to $G/P_x$. We show that $\gT$ acts on it naturally.

\begin{definition}
	We denote by $\text{Fr}(V|_x)$ the frame bundle given by the vector bundle $V|_x$. We define the (partial) flag bundle $\mathfrak{Fl}$ over $\Pic(X_{a})$ as the associated bundle 
	$\text{Fr}(V|_x)\times_G G/P_x$. Here $P_x$ is the parabolic subgroup corresponding to the parabolic structure at $x$. 
\end{definition}
By definition, $\mathfrak{Fl}$ parametrizes all the vector bundle filtrations with type given by $P_x$ on $V|_x$.

\btheorem\label{main5} For a generic $a\in \sH$, we have $(h^{W}_P)^{-1}(a)\cong\mathfrak{Fl}^\gT$. 
\etheorem

\bproof	
Since $\gT$ is a smooth commutative subgroup scheme of $\ssA ut(V|_x)$  (i.e., the group scheme of local automorphisms of vector bundle $V|_x$), $\gT$ acts naturally on the frame bundle $\text{Fr}(V|_x)$ which induces an action of $\gT$ on the associated bundle $\mathfrak{Fl}$.

Fiber-wise, fixed points are those $\{P_{i}\}\subset G/P_{x}$, such that $P_{i}\supset Z_{G}(\theta_{\Pic}|_{x})=T$. Since then $\theta_{\Pic}|_{x}\in\mathfrak{p}_{i}$, filtration determined by $P_{i}$ is compatible with $\theta_{\Pic}|_{x}$.

Conversely, $(\sE,\theta)$ lies in $(h^{W}_{P})^{-1}(a)$, meaning that $E$ is a line bundle over $X_{a}$, and has a filtration at $x$ compatible with $\theta_{\Pic}|_{x}$. A sub-bundle of parabolic sub-groups, $P'\subset \ssA ut(V|_{x})$, determines a filtration of $V|_x$. This filtration is compatible with $\theta_{\Pic}|_{x}$ if and only if $\theta_{\Pic}|_{x}\in\mathfrak{p}'$. Since $\theta_{\Pic}|_{x}$ is regular semisimple, its centralizer is a maximal torus contained in $P'$, i.e., $P'\supset \mathfrak{T}$. Thus it is a fixed point of $\mathfrak{T}$-action on $\mathfrak{Fl}$.\qed
\eproof

\begin{remark}
	The intuition of this theorem is as follows:
	filtrations coming from a parabolic structure should be compatible with the Higgs field at $x$, and thus they correspond to the fixed points of $\gT$ action on $\mathfrak{Fl}$.
\end{remark}

Denote by $W_x$ the Coxeter subgroup corresponding to the parabolic subgroup $P_x$, we have 

\blemma\label{compatible with theta} $\mathfrak{T}$ acts on $\mathfrak{Fl}$ over $\Pic(X_a)$, and the fixed points $\mathfrak{Fl}^{\mathfrak{T}}$ is a $W_x$ torsor on $\Pic(X_a)$.
\elemma
\bproof	
We know that $\gT\subset\sA ut(V|_{x})$ as a sub-group scheme, then $\gT$ acts on $\mathfrak{Fl}$.
Since fiber-wise we know that invariant points of $G/P_x$ under the action of $T\subset P_x$ is in bijection to $W_{x}$, we finish the proof.\qed
\eproof

Since here $G=\GL_{r}$, we can give a more explicit description. First, we denote $\pi^{-1}(x)$ by $\{y_{1},\ldots,y_{r}\}\subset X_{a}$. Then we restrict the universal line bundle $\ssP$ to each $\Pic(X_a)\times {y_i}$, and denote it by $\ssP|_{y_i}$. One has 
\beq\label{vxdecomp} V|_x\cong \oplus_{i=1}^r \ssP|_{y_i}\eeq since $\pi^{-1}(x)$ are $r$-distinct reduced points. 

Indeed, factors in the decomposition (\ref{vxdecomp}) are eigenspaces of $\theta_{\Pic}|_x$. So under this decomposition, $\theta_{\Pic}|_x$ is a direct sum of $\theta_{y_i}:\ssP_{y_i}\to\ssP_{y_i}$ and $\gT$ preserves the decomposition. To conclude:
\begin{corollary}
	The connected components $\pi_{0}((h^{W}_P)^{-1}(a))$ are in bijection to the Coxeter group $W_{x}$ associated with $P_x$.
\end{corollary}

\section{Global Nilpotent Cone of Strongly/Weakly Parabolic Hitchin Maps}\label{sec 6}
In this section, we study global properties of strongly/weakly parabolic Hitchin maps, i.e., flatness and surjectivity.
\begin{definition}
	We call $h_{P}^{-1}(0)$ (resp. $(h^{W}_{P})^{-1}(0)$) the strongly parabolic global nilpotent cone (resp. the weakly parabolic global nilpotent cone) denoted by $\sN il_{P}$(resp. $\sN il_{P}^{W}$).
\end{definition}

By Lemma \ref{dim formula} and formulas (\ref{dimofweakparabolichiggs}) and (\ref{dimofweakparabolichitchinbase}), we have  
\begin{align}&\label{dimestimate0.0} \dim(\text{fiber of }h_{P}) \geq \dim(\M_P)\\
	&\label{dimestimate0.1} \dim(\text{fiber of }h_P^W) \geq r^2(g-1)+1+\frac{r(r-1)\deg(D)}{2}.
\end{align}

\subsection{$\mathbb{G}_m$-actions on $\Higgs^W_P$ and $\Higgs_P$}

There is a natural $\mathbb{G}_m$-action on  the moduli stack of strongly/weakly parabolic Higgs bundles given pointwisely by $(\sE,\theta)\mapsto (\sE,t\theta), t\in \mathbb{G}_{m}$. It preserves stability and leaves Hilbert polynomials invariant. Thus it can be defined on the (semi-)stable locus and indeed the corresponding coarse moduli spaces, i.e., $\Higgs_{P}$ and $\Higgs^{W}_{P}$. This action was first studied by Hitchin \cite{Hit87}. It contains a lot of information of moduli spaces and Hitchin maps.

There is also a natural $\mathbb{G}_m$-action on $\sH$ and $\sH_{P}$: \[(a_1,a_2,\cdots,a_r)\mapsto (ta_1,t^2a_2,\cdots,t^ra_r),\]
and $h_{P}$, $h^{W}_{P}$ are equivariant under the $\mathbb{G}_m$-action. This can be used to show the flatness of Hitchin maps as in \cite{Gin01} if one has the dimension estimate of the global nilpotent cones. Over complex numbers, one can use the symplectic geometry as in \cite{Gin01} to show that the nilpotent cone is a Lagrangian substack of $T^\vee_{\text{Bun}_G}$ and then get the dimension estimate of the global nilpotent cone. But over a general algebraically closed field, especially in the positive characteristic case, one does not have the symplectic geometry. Instead we will use deformation theory in the next subsection to estimate the dimension of the strongly/weakly parabolic global nilpotent cones, which will imply the flatness of the strongly/weakly parabolic Hitchin maps.

\subsection{Dimensions of the Strongly/Weakly Parabolic Global Nilpotent Cones}

The study of infinitesimal deformations of strongly/weakly parabolic Higgs bundles was done in \cite{Yo95}, \cite{BR94} and \cite{Botta95}.  They discussed the deformation of a Higgs bundle along the whole moduli spaces and used this to give a description of the canonical symplectic form. In this subsection, we consider the deformation of a Higgs bundle along the reduced part of the irreducible components of the strongly/weakly parabolic global nilpotent cone and give the dimension estimates of the corresponding components.

\subsubsection{Strongly Parabolic Global Nilpotent Cone}

Recall that in \cite[Theorem 4.6]{Yo93C},\cite[Remark 5.1]{Yo95}, since we always assume the weight function is generic, see Assumption \ref{assm 0}, $\Higgs_P$ is a geometric quotient by an algebraic group $\PGL(V)$ of some $\PGL(V)$-scheme $\sQ$. Moreover, one has a universal family of stable strongly parabolic Higgs bundles with a surjection as in \cite[Remark 5.1]{Yo95} $$(\tilde{\sE},\tilde{\theta}) \text{ with a surjection }V\otimes_k \sO_{X_\sQ}\twoheadrightarrow \tilde{\sE}.$$ 

We denote the quotient map by $q:\sQ \to \Higgs_P$. Restricting the universal family $(V\otimes_k \sO_{X_{\sQ}}\twoheadrightarrow\tilde{\sE},\tilde{\theta})$ to $q^{-1}(\sN il_{P})$, we get:
$$\mathbb{U}_{\sN il_P}:=(V\otimes_k \sO_{X_{q^{-1}(\sN il_{P})}}\twoheadrightarrow\tilde{\sE},\tilde{\theta}).$$ 
For any scheme $S$ and flat family $(V\otimes_k \sO_S\twoheadrightarrow \sE_{S},\theta_S)$ of strongly parabolic Higgs bundles with $h_{P}((\sE_{S},\theta_{S}))=0$ on $S$, there is a map $\phi:S\to q^{-1}(\sN il_{P})$ such that $$(\id_X\times\phi)^*\mathbb{U}_{\sN il_P}\cong (V\otimes_k \sO_S\twoheadrightarrow \sE_{S},\theta_S).$$

To determine the dimension of $\sN il_P$, it is sufficient to calculate the dimension of each irreducible component with reduced structure. Restricting to any generic point $\eta$ of $\sN il_P$, $\theta_\eta:=\tilde{\theta}|_{q^{-1}(\eta)^\text{red}}$ gives a filtration $\{\Ker(\theta_\eta^i)\}$ of vector bundles of $\tilde{E}|_{q^{-1}(\eta)^\text{red}}$ (i.e., the graded terms are all vector bundles), because $X\times \eta$ is a curve. We can spread this out. 

\blemma\label{irrW} There exists an irreducible open subset $W\subset\sN il^{\text{red}}_P$ with generic point $\eta$, such that $\theta_W:=\tilde{\theta}|_{q^{-1}(W)^\text{red}}$ gives a filtration $\Ker(\theta_W^i)$ of vector bundles of $E_{W}:=\tilde{E}|_{q^{-1}(W)^\text{red}}$ over $X\times W$. 
\elemma

We fix some notations for filtered bundle maps. Let $E_1,E_2$ and $E$ be vector bundles on $X$ with decreasing filtrations by subbundles $K_i^\pt$, $i=1,2,\emptyset$ on each of them. We denote by \[\sH om^{\text{fil}}(E_1,E_2)(\text{resp. }\sE nd^{\text{fil}}(E))\] the coherent subsheaf of $\sH om(E_1,E_2)$ (resp. $\sE nd(E)$) consisting of those local homomorphisms preserving filtrations. And $\sH om^{\text{s-fil}}(E_1,E_2)$ (resp. $\sE nd^{\text{s-fil}}(E)$) consists of local homomorphisms $\phi$ such that $\phi(K^{j}_1|_{U})\subset K^{j+1}_2|_U$ (resp. $\phi(K^{j}|_{U})\subset K^{j+1}|_U$ ). 

Let us denote the decreasing filtration by:
\[K^\pt_W:E_{W}=K^0_W\supset K^1_W\supset\cdots\supset K^{r'}_W=0, \]
induced by $\{\Ker(\theta_W^i)\}$ on $E_{W}$. For $x\in D$, we also use $x$ to denote the closed immersion $\{x\}\times W\to X\times W$.

\blemma\label{finer flag} At each punctured point $x:\{x\}\times W\rightarrow X\times W$, $K^{\pt}_W|_{x}$ is a coarser flag than the parabolic structure $\sE_W|_{x}$, so \begin{small}\[SParHom(\sE_{W},\sE_{W}\otimes\omega_X(D))\cap Hom^{\text{s-fil}} (E_W,E_W\otimes\omega_X(D))
	=Hom^{\text{s-fil}} (E_W,E_W\otimes\omega_X(D)),
	\]
\end{small}
\noindent and $\theta_W\in Hom^{\text{s-fil}} (E_W,E_W\otimes\omega_X(D))$.
\elemma

\bproof $\theta_W$ is a strongly parabolic map then $F^{\sigma_{x}-j}(x)\subset K_W^{r'-j}$. In other words, $\sE_{W}|_{x}$ is a finer flag than $K^{\pt}_W|_{x}$. Besides, $\theta_W\in SParHom(\sE_{W},\sE_{W}\otimes\omega_X(D))\cap Hom^{\text{s-fil}} (E_W,E_W\otimes\omega_X(D))$ by definition.  \qed 
\eproof

Thus the nilpotent parabolic bundle $\sE_{W}$ has a filtration of vector bundles which does not depend on the surjection $V\otimes_k \sO_{q^{-1}(W)^{\text{red}}}\twoheadrightarrow E_{W}$. 


\btheorem\label{main 1} The space of infinitesimal deformations in $W$ of a nilpotent strongly parabolic Higgs bundle $(\sE,\theta)$, is canonically isomorphic to $\mathbb{H}^1(X,\sA^\pt)$. Here $\sA^\pt$ is the following complex of sheaves on $X$:
$$ 0\to \sP ar\sE nd(\sE)\cap\sE nd^{\text{fil}}(E)\xrightarrow{ad(\theta)}
(\sS\sP ar\sE nd(\sE)\cap \sE nd^{\text{s-fil}}(E) )\otimes\omega_X(D)\to 0,$$
which is isomorphic to
$$0\to \sP ar\sE nd(\sE)\cap \sE nd^{\text{fil}}(E)\xrightarrow{ad(\theta)}
\sE nd^{\text{s-fil}}(E) \otimes\omega_X(D)\to 0.$$
\etheorem

\bproof An infinitesimal deformation of a parabolic pair $(\sE,\theta)=u\in W$ is a flat family $(\boldsymbol{\sE},\boldsymbol{\theta})$ with $h_{P}((\boldsymbol{E},\boldsymbol{\theta}))=0$ parametrized by $\Spec(k[\epsilon]/\epsilon^2)$ together with a given isomorphism of $(\sE, \theta)$ with the specialization of $(\boldsymbol{\sE},\boldsymbol{\theta})$. By the local universal property of $\mathbb{U}_{\sN il_P}$, $(\boldsymbol{\sE},\boldsymbol{\theta})$ is the pull back of $(\tilde{\sE},\tilde{\theta})$ by a map $\phi:\Spec(k[\epsilon]/\epsilon^2)\to \sN il_P$. Moreover, if the deformation is inside $W$, then $\phi$ factors through $q^{-1}(W)^{\text{red}}$ and $(\boldsymbol{\sE},\boldsymbol{\theta})$ is a pull back of $(\sE_W,\theta_W)$.

Thus $$\boldsymbol{K}^{\pt}:=(id_X\times\phi)^*K^{\pt}_W$$ is a filtration on $\boldsymbol{E}$ such that $$\boldsymbol{\theta}\in SParHom(\boldsymbol{\sE},\boldsymbol{\sE}\otimes\omega_X(D))\cap Hom^{\text{s-fil}} (\boldsymbol{E},\boldsymbol{E}\otimes\omega_X(D)).$$ Since $K^{\pt}_W$ does not depend on the surjection $V\otimes_k \sO_{q^{-1}(W)^{red}}\twoheadrightarrow E_{W}$, $\boldsymbol{K}^{\pt}$ is uniquely determined by $(\boldsymbol{\sE},\boldsymbol{\theta})$.

Let us denote the projection by $\pi:X_{\epsilon}=X\times\Spec(k[\epsilon]/\epsilon^2)\to X$.
Tensoring $(\boldsymbol{\sE},\boldsymbol{K}^{\pt},\boldsymbol{\theta})$ with $$0\to (\epsilon)\to k[\epsilon]/\epsilon^2 \to k\to 0 ,$$ we have an extension of filtered strongly parabolic Higgs $\sO_{X_\epsilon}$-modules
\beq\label{ext2} 0\to  (\boldsymbol{\sE},\boldsymbol{K}^{\pt},\boldsymbol{\theta})(\epsilon)\to (\boldsymbol{\sE},\boldsymbol{K}^{\pt},\boldsymbol{\theta}) \to (\sE,K^{\pt},\theta)\to 0 .\eeq 
Pushing forward (\ref{ext2}) by $\pi$, we have an extension $$0\to  (\sE,K^{\pt},\theta)\to \pi_*(\boldsymbol{\sE},\boldsymbol{K}^{\pt},\boldsymbol{\theta}) \to (\sE,K^{\pt},\theta)\to 0$$ of locally free filtered strongly parabolic Higgs $\sO_X$-modules. The left inclusion will recover the $\sO_{X_\epsilon}$-module structure of $\pi_*(\boldsymbol{\sE},\boldsymbol{K}^{\pt},\boldsymbol{\theta})$. Thus $(\boldsymbol{\sE},\boldsymbol{K}^{\pt},\boldsymbol{\theta})$ is formally determined by an element in \[\Ext_{\text{fil}-par-Higgs-\sO_X}((\sE,K^{\pt},\theta),(\sE,K^{\pt},\theta)).\]
One can reinterpret  the extension class using \v Cech cohomology. Let $\sU=\{U_i\}_{i}$ be an affine finite covering of $X$, trivializing $E$ and all $K^j$. Then on each $U_i$, there is a splitting 
\[\phi_i:(\sE,K^{\pt})|_{U_i}\to(\boldsymbol{\sE},\boldsymbol{K}^{\pt})|_{U_i}\]
preserving the two compatible filtrations. The Higgs fields induce a filtered map \[\psi_i=\boldsymbol{\theta}\phi_i-\phi_i\theta:(\sE,K^{\pt})|_{U_i}\to(\sE,K^{\pt})|_{U_i}.\]

Thus the extension $(\boldsymbol{\sE},\boldsymbol{K}^{\pt})$ is given by a \v Cech 1-cocycle $(\phi_{ij}:=\phi_i-\phi_j)$ with values in $$\sP ar\sE nd(\sE)\cap \sE nd^{\text{fil}}(E)$$ and a \v Cech $0$-cochain $(\psi_i:=\boldsymbol{\theta}\phi_i-\phi_i\theta)$ with values in
$$ \sS \sP ar\sE nd(\sE)\otimes\omega_X(D)\cap \sE nd^{\text{s-fil}}(E)\otimes\omega_X(D).$$ 

One has \begin{align*}\delta(\phi_{ij})_{abc}&=\phi_{bc}-\phi_{ac}+\phi_{ab}\\
	&=\phi_b-\phi_c-\phi_a+\phi_c+\phi_a-\phi_b=0,
\end{align*} 
and \begin{align*}\delta(\psi_i)_{ab}&=\boldsymbol{\theta}\phi_{ab}-\phi_{ab}\theta\\
	&=\theta\phi_{ab}-\phi_{ab}\theta=ad(\theta)(\phi_{ab}).
\end{align*} 
It means that $((\phi_{ij}),(\psi_i))$ is a \v Cech 1-cocycle of the following complex of sheaves $\sA^\pt$ on $X$:
\[0\to \sP ar\sE nd(\sE)\cap\sE nd^{\text{fil}}(E)\xrightarrow{ad(\theta)} 
(\sS\sP ar\sE nd(\sE)\cap \sE nd^{\text{s-fil}}(E) )\otimes\omega_X(D)\to 0\]
If the extension is trivial, then $\phi_i=( 1, \phi_i')$ where $$\phi_i'\in  \sP ar\sE nd(\sE)\cap\sE nd^{\text{fil}}(E)$$ and $$\psi_i=\theta\phi'-\phi'\theta=ad(\theta)(\phi'_i).$$ Thus trivial extensions correspond to \v Cech 1-coboundaries of $\sA^\pt$. 

On the other hand, if we have a \v Cech 1-cocycle $((\phi_{ij}),(\psi_i))$, then use $\begin{bmatrix}I&\phi_{ij}\\ 0&I\end{bmatrix}$ to glue $$\{(\sE,K^\pt)|_{U_i}\oplus (\sE,K^\pt)|_{U_i}\}$$ with the local Higgs field $\begin{bmatrix}\theta&\psi_i\\0&\theta\end{bmatrix}$. One can check that the gluing condition of the local Higgs fields:
\[\begin{bmatrix}\theta&\psi_i\\0&\theta\end{bmatrix}\begin{bmatrix}I&\phi_{ij}\\ 0&I\end{bmatrix}=\begin{bmatrix}I&\phi_{ij}\\ 0&I\end{bmatrix}\begin{bmatrix}\theta&\psi_j\\0&\theta\end{bmatrix}\] is equivalent to the cocycle condition \[\delta(\psi_i)_{ab}=ad(\theta)(\phi_{ab}).\] 

If $((\phi_{ij}),(\psi_i))$ is a coboundary, i.e., \[((\phi_{ij}),(\psi_i))=\left( (\phi_i'-\phi_j'),(ad(\theta)(\phi_i')) \right).\] One can check :
\[\begin{bmatrix}-\phi'_i\\I\end{bmatrix}:(\sE,K^\pt)|_{U_i}\to(\sE,K^\pt)|_{U_i}\oplus (\sE,K^\pt)|_{U_i}\] can be glued to a global splitting of filtered Higgs bundles.\qed\eproof

The filtration $K^\pt$ of bundles on $E$ is equivalent as a $\sP$ reduction of $E$, where $\sP\subset \GL_r$ is the corresponding parabolic subgroup. We denote the principal $\sP$-bundle by $^\sP E$,  and let $U\subset \sP$ be the unipotent radical. We denote their Lie algebras by $\mathfrak{n}$, $\mathfrak{p}$. Thus $\sE nd^{\text{fil}}(E)\cong ad_{^\sP E}$, $\sE nd^{\text{s-fil}}(\sE)\cong ad_{^\sP E}(\mathfrak{n})$.

By Lemma \ref{finer flag}, we have $P_x\subset \sP$ for all $x\in D$. Denoting by $\mathfrak{p}_x$ the Lie algebra of $P_x$, we have 

\beq\label{n/u} 0\to \sP ar\sE nd(\sE)\cap \sE nd^{\text{fil}}(E)\to \sE nd^{\text{fil}}(E)\to \bigoplus_{x\in D}i_{x*}\text{ }\mathfrak{p}_{/\mathfrak{p}_x}\to 0.\eeq

According to (\ref{n/u}) we have \[0\to\sA^\pt\to\sA'^\pt\to \bigoplus_{x}i_{x*}\text{ }\mathfrak{p}_{/\mathfrak{p}_{x}}\to 0,\] where $\sA'^\bullet$ is
\[0\to \sE nd^{\text{fil}}(E)\xrightarrow{ad(\theta)}
\sE nd^{\text{s-fil}}(E) \otimes\omega_X(D)\to 0.\]

We also need the following Lemmas.
\blemma Let $\sP$ be a parabolic subgroup of $\GL_r$ and $U$ be its unipotent radical. We denote by $\mathfrak{p}$, $\g$ and $\mathfrak{n}$ their Lie algebras and $\sP$ acts on them by conjugation. One has $\mathfrak{n}^\vee\cong \mathfrak{g}/\mathfrak{p}$ as a $\sP$-linear representation.
\elemma
\bproof For $\g=\mathfrak{gl}_r$, the form $\beta:\g\times\g\to k\quad (A,B)\mapsto \tr(AB)$ is a non-degenerate $\GL_r$-equivariant bilinear form. Thus the isomorphism holds.\qed
\eproof
\blemma\label{adlemma} Let $E$ be a finite dimensional vector space over a field. $\theta:E\to E$ is a nilpotent endomorphism and $\mathfrak{p}$ is the parabolic algebra preserving the decreasing filtration given by $\{\Ker(\theta^{i})\}$. Let $\mathfrak{n}$ be the nilpotent radical. Then $ad(\theta):\mathfrak{p}\to \mathfrak{n}$ is surjective.
\elemma
\bproof Make an induction on the number of Levi factors of $\mathfrak{p}$.\qed\eproof

\bproposition\label{dimestimate1} We get the dimension estimate \begin{equation*}\dim_k(T_{u}W)=\dim_k(\mathbb{H}^1(X,\sA^\pt)) =\dim(\textbf{M}_P).
\end{equation*}
Thus any irreducible component of $\sN il_P$ has the same dimension as $\textbf{M}_P$. In particular, $\sN il_P$ is equi-dimensional.
\eproposition
\begin{proof}  One has \begin{align*}\chi(X,\sA'^\pt)&=\chi(\sE nd^{\text{fil}}(E))-\chi(\sE nd^{\text{s-fil}}(E) \otimes\omega_X(D))\\
		&=\chi(ad_{^\sP E}(\mathfrak{p}))-\chi(ad_{^\sP E}(\mathfrak{n}) \otimes\omega_X)-\deg(D)\cdot \dim_k(\mathfrak{n})\\
		&=\chi(ad_{^\sP E}(\mathfrak{p}))+\chi(ad_{^\sP E}(\mathfrak{g}/\mathfrak{p}))-\deg(D)\cdot \dim_k(\mathfrak{n})\\
		&= r^2(1-g)-\deg(D)\cdot \dim_k(\mathfrak{n}). 
	\end{align*}  
	Thus $
	\chi(X,\sA^\pt)= \chi(X,\sA'^\pt)-\sum_{x\in D} \dim_k(\mathfrak{p}_{/\mathfrak{p}_{x}})
	= r^2(1-g)-\sum_{x} \dim(G/P_x).$
	
	Elements in $\mathbb{H}^0(X,\sA^\pt)$ are those endomorphisms of $\sE$ commuting with $\theta$, then by the stability of $(\sE,\theta)$, we have $h^0(X,\sA^\pt)=1$. Base changing $\sA^\bullet$ to the generic point $\xi$ of $X$, we have $$\sA^\pt_\xi:\sE nd^{\text{fil}}(E_\xi)\xrightarrow{ad(\theta)_{\xi}} \sE nd^{\text{s-fil}}(E) \otimes\omega_X(D)_{\xi}\cong\sE nd^{\text{s-fil}}(E_{\xi}).$$ This map is surjective by Lemma \ref{adlemma}.  Thus $\tau^{\geq 1}\sA^{\pt}$ is supported on finitely many closed points of $X$ and $\mathbb{H}^2(X,\tau^{\geq 1}\sA^\pt)=0$. By
	$$\tau^{\leq 0}\sA^\pt\to\sA^\pt\to\tau^{\geq 1}\sA^\pt\xrightarrow{+1},$$ we have $h^2(X,\sA^\pt)=0$, and thus 
	$$\dim_k(T_{u}W)=\dim_k(\mathbb{H}^1(X,\sA^\pt))=1-\dim_k(\chi(X,\sA^\pt)) =\dim(\textbf{M}_P).$$\qed
\end{proof}

\btheorem\label{flat} If $\Higgs_P$ is smooth, then the strongly parabolic Hitchin map $h_P$ is flat and surjective.
\etheorem

\bproof By the dimension estimate of the strongly parabolic nilpotent cone, the remaining proof is similar as that in \cite[Corollary 1]{Gin01}. For any $s\in \sH_P-\{0\}$, $\bar{\mathbb{G}_m\cdot s}$ contains $0$. Since $h_P$ is equivariant under the $\mathbb{G}_m$-action, for each point $t\in \mathbb{G}_m\cdot s$, $h_P^{-1}(t)\cong h_P^{-1}(s)$. Thus $$\dim(h_P^{-1}(s))=\dim(\text{generic fiber of }h_P|_{\bar{\mathbb{G}_m\cdot s}})\leq \dim h_P^{-1}(0).$$ By (\ref{dimestimate0.0}), we have $\dim(h_P^{-1}(s))=\dim(\M_P)$ for any $s\in\sH_P$. Since $\Higgs^{W}_{P}$ and $\sH$ are both smooth, $h_P$ is flat.

Because all fibers are of dimension $\frac{1}{2}\Higgs_P=\dim\mathcal{H}_{P}$, $h_{P}$ is dominant. Since $h_P$ is projective by Proposition \ref{hproper}, it is surjective. \qed
\eproof

\subsubsection{Weakly Parabolic global nilpotent cone}

Let us compute the dimension of the weakly parabolic global nilpotent cone to show the flatness of $h^{W}_{P}$. However for weakly parabolic Higgs bundles $(\sE,\theta)$ in $\sN il_{P}^{W}$, the filtration $\{\Ker(\theta^{i})\}$ and the parabolic filtration are not compatible. As a result, it is not obvious to construct a complex dominating deformation within $\sN il^{W}_{P}$ as before.

We can still calculate $\dim\sN il^{W}_{P}$ by showing that there is a dominate map from a finite union of $\sN il^{W}_{B_i,\beta_i}$ to $\sN il^W_{P,\alpha}$. Here $B_i$ is a Borel quasi-parabolic structure refining $P$.

More precisely, for a Borel parabolic structure $(B,\beta)$, the weakly parabolic nilpotent cone coincides with the strongly parabolic nilpotent cone. Thus by Theorem \ref{dimestimate1}, $\sN il^W_{B,\beta}$ has the expected dimension $r^2(g-1)+1+\frac{r(r-1)\deg(D)}{2}$. 

For any generic point $\eta$ of $\sN il^W_{P,\alpha}$, by restricting the universal family on $\{D\}\times \eta$, it is not difficult to see there exists a Borel refinement $B_\eta$ of $P$, such that for general $(\sE,\theta)$ in the $\eta$-irreducible component of $\sN il^W_{P,\alpha}$, $\theta$ preserves the filtration given by $B_\eta$. 

One can choose a parabolic weight $\beta_\eta$ for each $B_\eta$, such that the stability is preserved after forgetting the parabolic structure from $(B_\eta,\beta_\eta)$ to $(P,\alpha)$. In other words, the forgetful map is well defined and restricts to $f_{\eta}:\sN il_{B_\eta,\beta_\eta}^W \to \sN il^W_{P,\alpha}$ which dominates the generic point $\eta$. Thus $\sqcup_{\eta}\sN il_{B_\eta,\beta_\eta}^W$ dominates $\sN il^W_{P,\alpha}$, and we conclude:

\btheorem The weakly parabolic nilpotent cone has dimension $$r^2(g-1)+1+\frac{r(r-1)\deg(D)}{2}.$$ Indeed, if $\Higgs_{P,\alpha}^W$ is smooth, the weakly parabolic Hitchin map $h_{P,\alpha}^W:\Higgs_{P,\alpha}^W \to\sH$ is flat. 
\etheorem

\subsection{Existence of very stable parabolic bundles}

\bdefinition We define that a system of parabolic Hodge bundles is a strongly parabolic Higgs bundle $(\sE,\theta)$ with a decomposition
$\sE\cong \oplus\sE^i$,
such that $\theta$ is decomposed as a direct sum of $\theta_i:\sE^i\to \sE^{i+1}\otimes \omega_X(D)$. Here $\{\sE^{i}\}$ are subbundles with induced parabolic structures.
\edefinition

If a strongly parabolic Higgs bundle $(\sE,\theta)$ is a fixed point of the $\mathbb{G}_m$-action, then it has the structure of a system of parabolic Hodge bundles. We have the following lemma similar to \cite[Lemma 4.1]{Sim92}, \cite[Theorem 8]{Sim90} and \cite[Theorem 5.2]{Yo95}:

\blemma\label{Gm fixed point} If the strongly parabolic Higgs bundle $(\sE,\theta)$ satisfies $(\sE,\theta)\cong (\sE,t\cdot\theta)$ for some $t\in \mathbb{G}_m(k)$ which is not a root of unity, then $\sE$ admits a structure as a system of parabolic Hodge bundles. In particular, if $\theta\neq 0$, then the decomposition $\sE\cong \oplus\sE^i $ given by the system of parabolic Hodge bundles is non-trivial.  
\elemma 

\begin{remark} One can conclude that given a strongly parabolic Higgs bundle $(\sE,\theta)$, if $\sE$ is stable and $\theta\neq 0$, it can not be fixed by the $\mathbb{G}_m$-action. 
	
\end{remark}

A section $s$ of $\sS \sP ar \sE nd(\sE)\otimes\omega_X(D)$ is nilpotent if $(\sE,s)\in \sN il_P$.
\bdefinition 
A stable parabolic bundle $\sE$ is said to be very stable if there is no non-zero nilpotent section in $H^0(X,\sS \sP ar \sE nd(\sE)\otimes\omega_X(D))$. 
\edefinition

\btheorem\label{very stable} The set of very stable parabolic bundles contains a non-empty Zariski open set in the moduli of stable parabolic bundles $\M_P$.
\etheorem

\bproof We denote by $N^0$ the open dense subset of $\Higgs_P$ consisting of $(\sE,\theta)$ such that $\sE$ is a stable parabolic vector bundle. Then $\pi: N^0\to \M_P$ by forgetting the Higgs fields is a well defined projection. 

$N^0$ is $\mathbb{G}_m$-equivariant in $\Higgs_P$, and $\pi$ is also $\mathbb{G}_m$-equivariant. Denote by $Z_1$ the set of $(\sE,\theta)$ with $\sE$ stable, $\theta$ nilpotent and nonzero. One observes that $Z_1\subset \sN il_P$, $Z_1$ is $\mathbb{G}_m$-equivariant and all the stable parabolic bundles which are not very stable are contained in $\pi(Z_1)$. 

Because $\sE$ is stable (cannot be decomposed) and $\theta$ is non-zero, then by Lemma \ref{Gm fixed point}, $\mathbb{G}_m$ acts freely on $Z_1$. Thus $Z_1/\mathbb{G}_m\twoheadrightarrow\pi(Z_1)$. One has $\dim(Z_1)\leq \dim(\sN il_P)=\dim(\M_P)$, so $$\dim(\pi(Z_1))\leq \dim(Z_1/\mathbb{G}_m)=\dim(\M_P)-1<\dim(\M_P).$$ Thus the set of very stable parabolic bundles contains a non-empty Zariski open set in $\M_P$.\qed
\eproof

\bcorollary 
For a generic choice of $a\in\mathcal{H}_P$, the natural forgetful map $h_{P}^{-1}(a)\dashrightarrow \M_P$ is a dominant rational map.
\ecorollary

\bproof 

By Theorem \ref{image parabolic}, we know that the image of $\Higgs_P$ is contained in $\mathcal{H}_P$.

Consider the following rational map:
\[\rho:\Higgs_P\dashrightarrow \sH_P\times \M_P\quad u\mapsto (h_P(u),\pi(u)).\]

By the existence of very stable parabolic vector bundles, i.e., there exists $(0,\sE)\in \mathcal{H}_{P}\times \M_P$ whose pre-image is $(0,\sE)\in \Higgs_P$. By Corollary \ref{dim formula}, we know that $\dim \Higgs_P=\dim\mathcal{H}_P+\dim \M_P$, which means that $\rho$
is generically finite. Thus $h_{P}^{-1}(a)\dashrightarrow \M_P$ is dominant, for generic $a\in\mathcal{H}_P$. \qed
\eproof

Using the similar method in \cite{PV85}, for example, as in \cite[Corollary 4.8]{WW21}, we can also show that the rational forgetful map $\text{Fg}:h^{-1}(a)\dashrightarrow \M_{P}$ is defined on an open sub-variety $U\subset h^{-1}(a)$ and $h^{-1}(a)\backslash U$ is of co-dimension $\geq 2$. It is well-known that there is a parabolic theta line bundle $\sL_{P}$ (which is not canonically defined) over $\M_{P}$. Then $\text{Fg}^{*}\sL_{P}$ can be extended to a line bundle over $h^{-1}(a)$, we still denote it by $\text{Fg}^{*}\sL_{P}$. To conclude:

\begin{corollary}
	For an $\ell\in\mathbb{Z}$, there is an embedding:
	\[
	H^{0}(\M_{P},\sL_{P}^{\otimes\ell})\hookrightarrow H^{0}(h^{-1}_P(a),\text{Fg}^{*}\sL_{P}^{\otimes\ell}).
	\]
\end{corollary} 

This is a generalization of that in \cite{BNR} to the parabolic case. It is interesting to see that the left-hand side vector space is also known as generalized parabolic theta functions of level $\ell$ (also referred to as conformal blocks in \cite{LS97}).

\section{Appendix}

In this appendix, we discuss singularities of generic spectral curves, along with ramifications. Since we may work over positive characteristics, a little bit more work is needed to use the Jacobian criterion. We assume $D=x$ and if $\text{char}(k)=2$, rank $r\geq 3$.  

\begin{lemma}\label{generic integrality}
	For a generic choice of $a\in \mathcal{H}_{P}$, the corresponding spectral curve $X_{a}$ is integral, totally ramified over $x$, and smooth elsewhere.
\end{lemma}
\bproof Since being integral is an open condition, similar as in \cite[Remark 3.1]{BNR}, we only need to show there exists $ a\in\mathcal{H}_P$, such that $X_{a}$ is integral.

Take $h_P((\sE,\theta))=\lambda^{r}+a_{r}=0$ with $a_{r}\in H^{0}(X, \omega^{\otimes r}((r-\gamma_{r})\cdot x))$.
The spectral curve $X_a$ is integral if $a_{r}$ is not an $m$-th power of some element in $\oplus_{i=1}^{r}H^{0}(X,\omega_X^{\otimes i}((i-\gamma_{i})x))$, for $m>1$. This is true for generic $a_{r}$. 

Since smoothness outside $x$ is an open condition, it is sufficient to find such a spectral curve.

When $\text{char}(k)\nmid r$, we take $h_P((\sE,\theta))=\lambda^{r}+a_{r}=0$. Due to the weak Bertini theorem, we can choose $a_{r}$ with only simple roots outside $x$. Applying the Jacobian criterion, $X_{a}$ is what we want.

When $\text{char}(k)\mid r$, we take the following equation:
$$h_P((\sE,\theta))=\lambda^{r}+a_{r-1}\lambda+a_{r}=0.$$

Then consider the following equations:
\[\left\{\begin{array}{rlcl}
	\lambda^{r}+a_{r-1}\lambda+&a_{r}&=&0\\
	&a_{r-1}&=&0\\
	a'_{r-1}\lambda+&a'_{r}&=&0
\end{array}\right. .\]
If $\text{char}(k)=2$, we assume $r\ge 3$. Then in this case, i.e., $\text{char}(k)\mid r$, $r$ is always greater than $2$. By the weak Bertini theorem, we can choose $a_{r-1}$ with only simple roots outside of $x$. Take $s\in H^0(X,\omega((1+\gamma_r-\gamma_{r-1})x))$ with zeros outside of $zero(a_{r-1})$, we can find $a_{r}=a_{r-1}\otimes s$ such that $zero(a_{r-1})\supset zero(a_{r})$, and $zero(a_{r-1})$ are simple zeros of $a_r$, then $X_{a}$ is smooth outside of $x$. \footnote{At $x$, $a_{r}$ will always have multiple zeros except for Borel type.} \qed
\eproof

\begin{lemma}\label{smooth, unram}
	For a generic $a\in \sH$, the corresponding spectral curve $X_{a}$ is smooth and $\pi_a:X_{a}\to X$ is unramified over $x$.
\end{lemma}
\bproof  
Smoothness is easy to show under the genericity condition. We will prove the second statement. The ramification divisor of $\pi_a$ is defined by the resultant. It is a divisor in the linear system of the line bundle $R:=\omega_X(x)^{\otimes r(r-1)}$. Considering the following morphism given by the resultant
$$\text{Res}: \sH\rightarrow H^{0}(X,R),\quad a\mapsto \text{Res}(a),$$ we have the codimension $1$ sub-space $$W:=H^{0}(X,R(-x))\subset H^{0}(X,R),$$ such that $\text{Res}(a)\in W$ if and only if $\pi_a$ is ramified over $x$. 

$\Res$ is a polynomial map so the image is a sub-variety. To prove our statement, we only need to find a particular $a$ so that $\pi_a$ is unramified over $x$. 

Consider the characteristic polynomial of the form
$\lambda^r+a_r$. In a neighbourhood of $x$, we can write it as $\lambda^r+b_r\cdot(\frac{dt}{t})^{\otimes r}$. Here $\frac{dt}{t}$ is the trivialization of $\omega(x)$ near $x$. By the Jacobian criterion, $\pi_a$ is unramified over $x$ if $b_r\in \sO_{X,x}$ is indecomposable. Taking $b_r=t$ and extending $t\cdot(\frac{dt}{t})^{\otimes r}$ to a global section $s$, we find $a=\lambda^r+s$ such that $\pi_a$ is unramified at $x$. \qed
\eproof
\vspace*{20pt}

\noindent\textbf{acknowledgements}. The authors thank Eduard Looijenga. Discussions with Eduard motivated our proof of the first main theorem and significantly affected the organization of this paper. The authors thank the referee for many helpful comments and suggestions to improve the paper. The authors also thank Bingyi Chen, Yifei Chen, H\'el\'ene Esnault, Yi Gu, Peigen Li, Yichen Qin, Junchao Shentu, Xiaotao Sun and Xiaokui Yang for helpful discussions.
	
	The work of Xiaoyu Su and Xueqing Wen was performed at Yau Mathematical Sciences Center and supported by Tsinghua Postdoctoral daily Foundation. The work of Bin Wang was performed at the Steklov International Mathematical Center, Moscow, Russia and supported by the Ministry of Science and Higher Education of the Russian Federation (agreement no.  075-15-2019-1614 ).

\bibliographystyle{alpha}
\bibliography{ref}

\end{document}